\documentclass[11pt]{amsart}
\usepackage[margin=3cm]{geometry}               
    \usepackage{graphicx}
\usepackage{tikz}
\usepackage{tikz-cd}
\usepackage{mathrsfs}
\usetikzlibrary{decorations.pathreplacing}
\usepackage{lineno, hyperref}
\usepackage{amssymb,amsfonts, amsmath, amsthm}
\usepackage{epstopdf}
\usepackage[shortlabels]{enumitem}
\usepackage{cancel}
\usepackage{thmtools} 
\usepackage{mathtools} 
\usepackage{mathrsfs}
\usepackage{yhmath}
\usepackage{amsmath}
\usepackage[english]{babel}
\usepackage[toc,page]{appendix}
\usepackage{float}
\restylefloat{table}
 \usepackage{hyperref}
  \hypersetup{
colorlinks=true,
linkcolor=cyan,
citecolor=cyan
}
\usepackage{amsbsy}

\swapnumbers
\newtheorem*{theorem*}{Theorem}
\newtheorem*{question*}{Question}
\newtheorem*{corollary*}{Corollary}
\newtheorem*{theoremA}{Theorem A}
\newtheorem*{theoremB}{Theorem B}
\newtheorem*{theoremC}{Theorem C}

\newtheorem*{main-theorem}{Main Theorem}

\newtheorem{theorem}[subsection]{Theorem}
\newtheorem{lemma}[theorem]{Lemma}
\newtheorem{fact}[theorem]{Fact}
\newtheorem{claim}[theorem]{Claim}
\newtheorem*{claim*}{Claim}
\newtheorem{proposition}[theorem]{Proposition}
\newtheorem{corollary}[theorem]{Corollary}
\newtheorem{conjecture}[theorem]{Conjecture}
\newtheorem{example}[theorem]{Example}

\theoremstyle{definition}
\newtheorem{definition}[theorem]{Definition}

\newtheorem{notation}[theorem]{Notation}

\newtheorem{remark}[theorem]{Remark}
\newtheorem{lem-def}[theorem]{Lemma-definition}

\newtheorem{question}[theorem]{Question}

\newcommand{\sep}{^{\rm sep}}

\newcommand{\N}{\mathbb N}

\newcommand{\F}{\mathbb F}
\newcommand{\cF}{\mathscr F}

\newcommand{\cM}{\mathcal M}
\newcommand{\cU}{\mathcal U}
\newcommand{\Fp}{\F_p}
\newcommand{\OO}{\mathcal{O}}
\newcommand{\cL}{\mathcal{L}}
\newcommand{\cC}{\mathcal{C}}
\newcommand{\cA}{\mathcal{A}}
\newcommand{\cB}{\mathcal{B}}
\newcommand{\cD}{\mathcal{D}}

\newcommand{\cK}{\mathcal{K}}
\newcommand{\res}{\mathrm{res}}

\newcommand{\Def}{\mathrm{def}}

\newcommand{\ACF}{\mathrm{ACF}}
\newcommand{\SCF}{\mathrm{SCF}}
\newcommand{\SCVF}{\mathrm{SCVF}}
\newcommand{\ACVF}{\mathrm{ACVF}}
\newcommand{\DE}{\mathrm{def}}
\newcommand{\DOAG}{\mathrm{DOAG}}

\newcommand{\Th}{\mathrm{Th}}

\newcommand{\bdd}{\mathrm{bdd}}
\newcommand{\OOnv}{\mathrm{conv}}

\newcommand{\tp}{\mathrm{tp}}

\newcommand{\id}{\mathrm{id}}

\newcommand{\SE}{\mathcal{S}\mathcal{E}}

\newcommand{\BP}{\mathrm{bp}}
\newcommand{\triv}{\mathrm{triv}}
\newcommand{\Aut}{\mathrm{Aut}}
\newcommand{\vd}{\mathrm{vd}}
\newcommand{\ld}{\mathrm{ld}}
\newcommand{\alg}{\mathrm{alg}}
\newcommand{\qftp}{\mathrm{qftp}}

\newcommand{\Ring}{\mathrm{ring}}
\newcommand{\og}{\mathrm{og}}
\newcommand{\trdeg}{\mathrm{trdeg}}

\newcommand{\cl}{\mathrm{cl}}

\newcommand{\Res}{\mathrm{Res}}
\newcommand{\VF}{\mathbf{VF}}
\newcommand{\RF}{\mathbf{k}}
\newcommand{\VG}{\mathbf{\Gamma}}
\newcommand{\VGi}{\mathbf{\Gamma_\infty}}
\newcommand{\cLr}{\cL_{\mathrm{ring}}}
\newcommand{\cLog}{\cL_{\mathrm{og}}}

\newcommand{\cLgk}{\cL_{\Gamma k}}

\DeclareMathAlphabet{\mathscrbf}{OMS}{mdugm}{b}{n}

\DeclareMathAlphabet{\mathscrbf}{OMS}{mdugm}{b}{n}

\newcommand{\elres}{\preccurlyeq}

\newcommand{\sepa}{\text{\tiny$\mathcal{S}\mathcal{E}$}}

\newcommand{\forkindep}[1][]{%
  \mathrel{
    \mathop{
      \vcenter{
        \hbox{\oalign{\noalign{\kern-.3ex}\hfil$\vert$\hfil\cr
              \noalign{\kern-.7ex}
              $\smile$\cr\noalign{\kern-.3ex}}}
      }
    }\displaylimits_{#1}
  }
}

\def\Ind#1#2{#1\setbox0=\hbox{$#1x$}\kern\wd0\hbox to 0pt{\hss$#1\mid$\hss}
\lower.9\ht0\hbox to 0pt{\hss$#1\smile$\hss}\kern\wd0}
\def\ind{\mathop{\mathpalette\Ind{}}}
\def\notind#1#2{#1\setbox0=\hbox{$#1x$}\kern\wd0\hbox to 0pt{\mathchardef
\nn=12854\hss$#1\nn$\kern1.4\wd0\hss}\hbox to
0pt{\hss$#1\mid$\hss}\lower.9\ht0 \hbox to
0pt{\hss$#1\smile$\hss}\kern\wd0}

\makeatletter
\def\widebreve{\mathpalette\wide@breve}
\def\wide@breve#1#2{\sbox\z@{$#1#2$}%
     \mathop{\vbox{\m@th\ialign{##\crcr
\kern0.08em\brevefill#1{0.8\wd\z@}\crcr\noalign{\nointerlineskip}%
                    $\hss#1#2\hss$\crcr}}}\limits}
\def\brevefill#1#2{$\m@th\sbox\tw@{$#1($}%
  \hss\resizebox{#2}{\wd\tw@}{\rotatebox[origin=c]{90}{\upshape(}}\hss$}
\makeatletter

\title[On separable pairs of fields and beautiful pairs of SCVF]{On separable pairs of fields and beautiful pairs of separably closed valued fields}

\author[Zhengqing He]{Zhengqing He}
\address{Universit\"{a}t Duisburg-Essen, Fakult\"{a}t für Mathematik, Mathematikcarr\'ee, Thea-Leymann-Str. 9, D-45127 Essen, Germany}
\email{zhengqing.he@stud.uni-due.de}
\thanks{}

\author[Martin Hils]{Martin Hils}
\address{Institut f\"{u}r Mathematische Logik und Grundlagenforschung, Universit\"{a}t M\"{u}nster, Einsteinstr. 62, D-48149 M\"{u}nster, Germany}
\email{hils@uni-muenster.de}
\thanks{MH was partially supported by the German Research Foundation (DFG) under Germany’s Excellence Strategy EXC 2044/2 - 390685587, Mathematics M\"{u}nster:
Dynamics-Geometry-Structure.}



\date{\today }

\begin{document}

\bibliographystyle{plain}

\begin{abstract} We continue the study of beautiful pairs of henselian valued fields, initiated in joint work of the second author with Cubides Kovacsics and Ye. As our main result, we show that separably closed non-trivially valued fields of arbitrary imperfection degree admit beautiful pairs associated to natural subclasses of all bounded definable types, most notably for all definable types orthogonal to the value group and for all bounded definable types. These beautiful pairs are elementary pairs satisfying certain natural axioms. They satisfy beauty transfer, i.e., saturated models of their theory are again beautiful. This yields strict pro-definability of the space of the definable types from any of these classes, in the theory SCVF of separably closed non-trivially valued fields. 

We also show a separable  analog of a classical result of Keisler's, namely that the first order theory of an infinite separable extension $L/K$ of fields -- considered in the pairs language and where the large field $L$ is separably closed, is determined by the theory of the small field $K$ together with the imperfection degree of $L$. 
\end{abstract}

\maketitle

\section*{{\color{red}A fatal error making incorrect certain cases of our main results}}

{\color{red}There is a fatal error in our proof of Proposition~\ref{prop:SepMaxima->separated}. We thank Ruohan Hu for having brought this error to our attention. Actually, assuming the hypotheses stated in Proposition~\ref{prop:SepMaxima->separated}, the extension $L/K$ is $vs$-defectless (i.e., the conclusion holds) precisely when $\VG(K)$ is cofinal in $\VG(L)$. This affects our main results in Section~7 and Section~8. Given the flaw in Proposition~\ref{prop:SepMaxima->separated}, our main results only hold under the additional assumption that the definable types in question are {\bf bounded}.

We note a second error, found by AI: If the models of the complete theory $T$ considered in Section~5 are of infinite degree of imperfection, contrarily to what we have written, the theory $\SCF_T$ is already complete, since any saturated model $(K,P(K))$ of $\SCF_T$ satisfies $e(K/P(K))=\infty$. This affects many results in sections 5-8, which take a different (and actually simpler) form.

The present manuscript is still the original version -- apart from the abstract -- and in particular it contains the two mistakes we have mentioned. We are currently rewriting the paper, but we prefer to make these issues known already now.
}

\section{Introduction}
Since the early days of model theory, pairs of models have been studied extensively, as objects in their own right but also in order to devise powerful tools for the study of the models themselves. Given an $\cL$-theory $T$ and $M\models T$, one names a substructure of $M$ with a new unary predicate $P$ and considers the $\cL_P$-structure $(M,P(M))$, where $\cL_P:=\cL\cup\{P\}$, often in order to reveal more subtle properties of the original $T$ and/or $M$. In the model theory of algebraic structures, the study of pairs was started  by A. Robinson in \cite{Robinson}, who showed, e.g., that the theory of proper pairs of algebraically closed fields of some fixed characteristic is complete. This was subsequently generalized by Keisler in \cite{Kei64} who proved that given a complete theory $T$ of fields, the theory $\ACF_T$ of pairs of fields $(K,P(K))$ with $P(K)\models T$, $K=K^{\alg}$ and $[K:P(K)]=\infty$ is complete.

Poizat \cite{poizat} used the complete theory $T_{\BP}$ of certain elementary pairs (called \emph{beautiful pairs}) of models of a complete stably theory $T$ as an approach to study nfcp and phenomena in connection with definable types in $T$. In modern terminology, he showed that $T$ does not have the finite cover property (i.e., is nfcp) if and only if every sufficiently saturated model of $T_{\BP}$ is beautiful -- which means $T$ has \emph{beauty transfer} -- if and only if the space of all global definable types may be naturally identified with a strict pro-definable set, i.e., with a projective limit of a family of definable sets with surjective definable transition functions. 

In their celebrated work \cite{HL} on topological tameness in Berkovich analytifications of algebraic varieties, Hrushovski and Loeser show that given an algebraic variety (or more generally a definable set) in a model of the theory $\ACVF$ of algebraically closed non-trivially valued fields, the space $\widehat{X}$ of all global definable types on $X$ which are orthogonal to the value group is strict pro-definable. This is a key fact in \cite{HL}. In their recent work \cite{CuHiYe26}, Cubides-Kovacsics, Ye and the second author have developped a theory of beautiful pairs outside the stable context, which in the case of $\ACVF$ yields strict pro-definability of other spaces of definable types as well, in particular of the space $\widetilde{X}$ of all bounded definable types and the space $S^{\Def}_X$ of all definable types. The formalism of pairs makes it possible to bring to bear the Ax-Kochen-Ershov philosophy in this context, reducing, e.g.,  the axiomatization and the verification of beauty transfer of the theory of beautiful pairs to the residue field and value group pairs. In a model of ACVF, the residue field $\RF$ and the value group $\VG$ are pure stably embedded model of ACF and of the theory DOAG of divisible non-trivial ordered abelian groups, respectively. There are 4 completions of the theory DOAG$_{\SE}$ of stably embedded elementary pairs $(M,P(M))$ of models of DOAG, determined by whether $M$ contains infinitesimal and/or infinite elements with respect to $P(M)$, and there are only two such completions for the residue field pairs, by the result of Robinson mentioned above. If $T$ is a completion of ACVF, the completions of $T_{\SE}$ are determined by the ones for the residue field and value group pairs, and they all correspond to theories of $\cK$-beautiful pairs, having beauty transfer, for an appropriate choice of \emph{natural class} 
$\cK$.

In this paper we provide \emph{analogs in the separable category} of the results we mentioned, working with separably closed (valued) fields instead of algebraically closed ones, and considering pairs of fields $(K,P(K))$ where $K/P(K)$ is a separable extension. Recall that the theory $\SCF_p$ of separably closed fields of characteristic $p>0$ is stable, and its completions are determined by fixing the impferfection degree $e\in\mathbb{N}\cup\{\infty\}$. Similarly, the theory $\SCVF_p$ of non-trivially valued separably closed fields of characteristic $p$ is NIP, and again fixing the imperfection degree yields a complete theory.

We first work in pure fields and show the following analog of Keisler's result.

   \begin{theoremA}[Theorem~\ref{T:Keisler}]
        For any complete theory $T$ of fields of characteristic $p>0$ we have the following.
        \begin{enumerate}
            \item The completions of $\SCF_T$ are given by $\SCF_{T,e}$, for $e\in\N\cup\{\infty\}$.
        \item For a fixed $e\in\N\cup\{\infty\}$, any good isomorphism between good substructures of models of $\SCF_{T,e}$ is an elementary map
        \item In $\SCF_{T,e}$, $P$ is stably embedded with induced structure a pure field.
        \end{enumerate}
    \end{theoremA}

    Here, $\SCF_T$ denotes the theory of separable pairs $(K,P(K))$ with $P(K)\models T$, $K=K^{\mathrm{sep}}$ and $[K:P(K)]=\infty$. In $\SCF_{T,e}$ one requires in addition that the relative imperfection degree of $K$ over $P(K)$ equals $e$.

Our main work is then done in the valued category. As a key technical result we establish an Embedding Lemma (Lemma~\ref{lem:Embedding}) for $vs$-defectless separable pairs. We first deduce from it a general result on $vs$-defectless separable (not necessarily elementary) pairs of models of $\SCVF_{p}$ (Theorem~\ref{thm:SCVF-Pair-Complete}), before we specialize to elementary pairs, thus obtaining our main result on beautiful pairs in $\SCVF_{p,e}$, for finite and infinite $e$ (Theorem~\ref{thm:beauty-SCVF}), whose proof is structurally very similar to the proof from \cite{CuHiYe26} in the case of $\ACVF$. As a consequence, we obtain the following two main results of the paper, which were the initial motivation for our work.

\begin{theoremB}[Corollary~\ref{cor:completions-SE}]
Let $(\SCVF_{p,e})_{\SE}$ be the $\cL_P$-theory of stably embedded elementary pairs of models of $\SCVF_{p,e}$. 
\begin{enumerate}
    \item If $e<\infty$, $(\SCVF_{p,e})_{\SE}$ has 8 completions, namely the theory of the trivial pair and the theories $\SCVF_{p,(e,0)}(T_{\RF},T_{\VG})$, for $T_{\RF}$ and $T_{\VG}$ not both trivial. Each completion corresponds to the theory of $\cK$-beautiful pairs for some natural class $\cK$ having AP, EP and beauty transfer.
    \item The completions of $(\SCVF_{p,e})_{\SE}$ are given by the theory of the trivial pair and by the theories $\SCVF_{p,(\infty,e_P)}(T_{\RF},T_{\VG})$, for $e_P\in\mathbb{N}\cup\{\infty\}$, and $T_{\RF}$ and $T_{\VG}$ not both trivial. Each completion corresponds to the theory of $\cK$-beautiful pairs for some natural class $\cK$ having AP, EP and beauty transfer.
\end{enumerate}
\end{theoremB}

In \cite{HKR2016}, given a definable set $X$ in $\SCVF_{p,e}$ for finite $e$, strict pro-definability of the space $\widehat{X}$ of all generically stable types on $X$ is established, using an abstract version of the proof of the corresponding result for $\ACVF$, which plays a crucial role in the work \cite{HL} by Hrushovski-Loeser. The following is a vast generalization of the result from \cite{HKR2016}.

\begin{theoremC}[Corollary~\ref{Cor:strictpro}]
    Let $\cK=\cK_{p,(e,e_P)}(\cK_{\VG},\cK_{\RF})$ be a natural class of definable types in the theory $\SCVF_{p,e}$, and let $X$ be an $A$-definable set, where  $A\subseteq\cU\models\SCVF_{p,e}$. Then the space $\cF_X$ carries naturally the structure of a strict $A$-pro-definable set. In particular, the space $S^{\DE}_{X}$ of all definable types, the space $\widetilde{X}$ of all bounded definable types and the space $\widehat{X}$ of all generically stable definable types are strict pro-definable.\qedhere
\end{theoremC}

These two results give a complete answer  -- in the affirmative -- to \cite[Question~9.3.6]{CuHiYe26}, and this even in arbitrary degree of imperfection.

\subsection*{Structure of the article} 
In Section~\ref{S:prelim-sep} we provide preliminaries on separable field extensions and the model theory of separably closed fields. Section~\ref{S:Basic-SCVF} is devoted to basic results on separably closed valued fields and their model theory, and in Section~\ref{S:beauty} we quickly present the formalism of beautiful pairs from \cite{CuHiYe26}, with particular emphasis on the case of $\ACVF$. The separable analogue of Keisler's result is then treated in Section~\ref{S:Keisler}. 

The last three sections deal with $vs$-defectless pairs in $\SCVF$. In Section~\ref{S:Dom-SCVF} we prove a domination result and deduce from it an algebraic characterization of definable types in $\SCVF$, while the study of separable $vs$-defectless pairs of models of $\SCVF$ is the content of Section~\ref{S:sep-vs-defectless}. Finally, we show in Section~\ref{S:SCVF-BP} our main result on beautiful pairs in $\SCVF$ and deduce from it Theorem~B and Theorem~C..

\subsection*{Acknowledgements} A substantial portion of the results of this article was obtained in the master thesis of the first author at the University of M\"unster, supervised by the second author. The second author would like to thank Silvain Rideau-Kikuchi for helpful discussions.

\section{Preliminaries on separable field extensions and separably closed fields}\label{S:prelim-sep}

\subsection*{Notation and conventions}We use standard notations from model theory, as may be found, e.g., in \cite{tent-ziegler}. 

Given fields $K\subseteq L,M\subseteq\cU$, we write $L\ind_K^{ld}M$ when $L$ and $M$ are \emph{linearly disjoint} over $K$, i.e., when any $K$-linearly independent tuple from $L$ is $M$-linearly independent. 

We now fix a prime $p$ and consider fields of characteristic $p$.

\begin{itemize}
    \item For a field $K$ and $n\in\mathbb{Z}$ we set $K^{p^n}:=\{x^{p^n}\mid x\in K\}$, which is a subfield of $K^{\alg}$. We set $K^{p^\infty}:=\bigcap_{n\geq0} K^{p^n}$ and $K^{p^{-\infty}}:=\bigcup_{n\geq0} K^{p^{-n}}$.
    \item A field extension $L/K$ is called \emph{separable} if $L^p\ind^{ld}_{K^p}K$ (inside $L$), equivalently if $L\ind^{ld}_KK^{p^{-\infty}}$ (inside $L^{\alg}$).
    \item For $A\subseteq K$, we set $\cl_K^p(A):=K^p(A)=K^p[A]$, called the \emph{$p$-closure of $A$ in $K$}. It is well known that $\cl_K^p$ defines a pregeometry on $K$, in particular the Steinitz Exchange Lemma holds for $\cl_K^p$. We thus get corresponding notions of \emph{$p$-independence}, \emph{$p$-basis} and \emph{$p$-dimension}. For $A,B,C\subseteq K$ we write $A\ind^p_BC$ if $A$ is $p$-independent from $C$ over $B$ (inside $K$).
    \item The \emph{imperfection degree} of $K$ is given by
    \[e(K):=\begin{cases}
        n & \,\text{if $K$ has a finite $p$-basis of cardinality $n$}\\
\infty & \, \text{else}\end{cases}\]
For an extension $K/K_0$, the relative imperfection degree $e(K/K_0)$ is similarly defined.
\item If $\mathbf{b}=(b_,\ldots,b_n)$ is a $p$-independent tuple from $K$ and $a\in\cl_K^p(\{b_1,\ldots,b_n\})=K^p[b_1,\ldots,b_n]$, it follows from the definitions that $a$ can be uniquely written as 
\begin{equation}\label{eq:lambda}
a=\sum_\nu a_\nu^p\mathbf{b}^\nu,    
\end{equation}
where $a_\nu\in K$, $\nu$ runs over $p^n=\{0,\ldots,p-1\}^n$ and $\mathbf{b}^\nu:=\prod_ib_i^{\nu_i}$.
\item Let $\cLr^{\lambda}:=\cLr\cup\{\lambda_{n,\nu}\mid n\geq0, \nu\in p^n\}$, where $\lambda_{n,\nu}$ is an $(n+1)$-ary function symbol. Using the uniqueness observed in the previous item, one may turn any field $K$ of characteristic $p$ into an $\cLr^{\lambda}$-structure by defining 
\[\lambda_{0,\emptyset}(a):=\begin{cases}
 a^{1/p}& \,\text{if $a\in K^p$}\\
 0 & \,\text{else}
\end{cases}\]
and
\[\lambda_{n,\nu}(a;\mathbf{b}):=\begin{cases}
 0 & \, \text{if $\mathbf{b}$ is not $p$-independent or $a\not\in\cl_K^p(\{b_1,\ldots,b_n\})$}\\
 a_\nu& \,\text{if $\mathbf{b}$ is $p$-independent and $a=\sum_\mu a_\mu^p\mathbf{b}^\mu$ as in \eqref{eq:lambda}}
\end{cases}\]
for $n>0$ and $\nu\in p^n$. The functions $\lambda_{n,\nu}$ are called the \emph{parametrized $p$-coordinate functions}. 
\end{itemize}

\begin{fact}\label{F:characterize-sep-ext}
    For an extension $L/K$ of fields of characteristic $p$, the following are equivalent:
    \begin{enumerate}
        \item $L/K$ is separable.
        \item There exists a $p$-basis of $K$ which is $p$-independent in $L$.
        \item Any tuple from $K$ which is $p$-independent in $K$ is $p$-independent in $L$.
        \item If $K$ and $L$ are considered as $\cLr^{\lambda}$-structures, then $K$ is a substructure of $L$.
        \item For every field $L_0$ with $K\subseteq L_0\subseteq L$ and $L_0/K$ finitely generated there exists a transcendence basis $T$ of $L_0$ over $K$ such that $L_0/K(T)$ is separable algebraic.
    \end{enumerate}
\end{fact}

\begin{proof}
    The equivalences of (1), (2), (3) and (4) are straight forward, as is (5)$\Rightarrow$(1). The non-trivial implication (1)$\Rightarrow$(5) is due to Mac Lane \cite{Mac39}.
\end{proof}

The following observation is useful when passing from a subring to its field of fractions.

\begin{lemma}\label{L:lambda-subfield}
    Let $K$ be a field of characteristic $p$ and $R\subseteq K$ an $\cLr^{\lambda}$-substructure of $K$ such that $R\not\subseteq K^{p^{\infty}}$. Then $R$ is a field.
\end{lemma}

\begin{proof}
If $R\subseteq K^p$, iteratively applying $\lambda_{0,\emptyset}$ would yield $R\subseteq K^{p^{\infty}}$, contradicting the assumption.
Thus we have $R\not\subseteq K^p$. Let $a_0\in R\backslash K^p$, and $\nu:1\to p$, $0\mapsto 1$. Then $\lambda_{1,\nu}(a_0;a_0^{p+1})=1/a_0\in R$. It follows that $R$ is a field, as for $b\in R\cap K^p$, we have $a_0\cdot b\in R\backslash K^p$ and thus $1/b=a_0\cdot\frac{1}{a_0\cdot b}\in R$.   
\end{proof}

For $e\in\mathbb{N}\cup\{\infty\}$ we denote by $\SCF_{p,e}$ the $\cLr$-theory of separably closed fields of characteristic $p$ and degree of imperfection $e$. By abuse of notation, we also write $\SCF_{p,e}$ for the corresponding $\cLr^{\lambda}$-theory, which is a definitional expansion. The following result is classical.

\begin{fact}\label{F:QE-SCF}
The theory $\SCF_{p,e}$ has quantifier elimination in $\cLr^{\lambda}$ and is complete. It is stable, and it is non-superstable if and only if $e>0$.
\end{fact}

\section{Basic results on separably closed valued fields}\label{S:Basic-SCVF}

\subsection*{Languages and theories of valued fields}\label{sec:languages}
For a valued field $(K,v)$, we let $\VG(K)$ denote its value group $v(K^\times)$, $\mathcal{O}(K)$ its valuation ring with unique maximal ideal $\cM(K)$, $\RF(K)=\mathcal{O}(K)/\cM(K)$ its residue field and $\res\colon \mathcal{O}(K)\to \RF(K)$ the residue map. By default, unless otherwise stated, by a valued field extension $K\subseteq L$ we implicitly mean $(K, v)\subseteq (L,v)$ and always use $v$ for the underlying valuation. When working with multi-sorted structures, we will also use the notation $\VF(K)$ to denote the field $K$. 

\medskip

We will mainly work in the following languages and theories of valued fields. 
\begin{itemize}[leftmargin=*]
    \item The 3-sorted language of valued fields $\cLgk$ has sorts $\VF, \RF$, and $\VGi$, where $\VF$ and $\RF$ are equipped with $\cLr$, $\VGi$ is equipped with $\cLog$ together with a new constant symbol for $\infty$, and we have additional function symbols for the valuation $v\colon \VF\rightarrow\VGi$ and the following residue map $\Res\colon \VF^2\rightarrow \RF$ interpreted in a valued field $(K,v)$ by     
\[
\Res(x,y)\coloneqq
\begin{cases}
\res(\frac{x}{y}) & \text{ if } \infty\neq v(y)\leqslant v(x);\\
0 & \text{otherwise.}
\end{cases}
\]
We let $\VG$ be $\VGi\setminus \{\infty\}$. We denote by $\ACVF$ the $\cLgk$-theory of algebraically closed non-trivially valued fields.     

\item Fix a prime $p>0$. We will treat valued fields of characteristic $p$ in the language $\cLgk^\lambda:=\cLgk\cup\{\lambda_{n,\nu}\mid n\in\mathbb{N}, \nu\in p^n\}$, where the parametrized $p$-coordinate functions are functions on sort $\VF$. We denote by $\SCVF_p$  the $\cLgk^{\lambda}$-theory of separably closed non-trivially valued fields of characteristic $p$. For $e\in\mathbb{N}\cup\{\infty\}$, $\SCVF_{p,e}$ denotes the theory $\SCVF_p$ together with axioms stipulating that the imperfection degree of the valued field is $e$.  
\end{itemize}

The following elementary fact shows that the auxiliary sorts in models of $\SCVF_p$ are those from their algebraic closure.

\begin{fact}[\mbox{\cite[Lemma~5.2.5 \& Remark~5.2.6]{HonPhD}}]\label{F:dense}
Let $(K,v)\models\SCVF_p$. Then $K^{p^n}$ is $v$-dense in $K^{\alg}$ for every $n\geq0$. In particular, the extension $K^{\alg}/K^{p^n}$ is immediate. 

In case $(K,v)$ is $\omega$-saturated, $K^{p^\infty}$ is $v$-dense in $K^{\alg}$.
\end{fact}

\subsection*{Quantifier elimination in separably closed valued fields} Let us first state the result in the algebraically closed case.

\begin{fact}[\cite{HHM1}]\label{F:QE-ACVF}
The theory $\ACVF$ eliminates quantifiers in $\cLgk$. Its completions are given by fixing the characterics of the valued field and of the residue field. In a model $M$ of $\ACVF$, the residue field $\RF(M)$ is stably embedded with induced structure a pure algebraically closed field, and the value group $\VG(M)$ is stably embedded with induced structure a pure model of $\DOAG$ (with a constant for $v(p)$ in mixed characteristic $(0,p)$). Moreover, $\RF(M)\perp\VG(M)$.\qedhere
\end{fact}

The following result is well known. We include a proof for convenience, as we could not find it stated in this form in the literature. A modern proof of the corresponding QE result in the one-sorted language $\cL_{\mathrm{div}}^\lambda$ may be found in \cite{Hon16}.

\begin{proposition}\label{prop:QE-SCVF}
Fix a prime $p$ and $e\in \mathbb{N}\cup\{\infty\}$. Then the following holds.
\begin{enumerate}
\item $\SCVF_{p,e}$ eliminates quantifiers in $\cLgk^{\lambda}$ and is complete. 
    \item In $M\models \SCVF_{p,e}$, the residue field $\RF(M)$ is stably embedded with induced structure a pure model of $\ACF_p$, and the value group $\VG(M)$ is stably embedded with induced structure a pure model of $\DOAG$. Moreover, $\RF(M)\perp\VG(M)$.
\end{enumerate}
\end{proposition}

\begin{proof}
Part (2) follows from (1) using syntactic considerations. 

We will now prove (1). Let $M,N$ be models of $\SCVF_{p,e}^\lambda$
	with $M$ countable and $N$  $\aleph_1$-saturated and
	let $A$ be a common substructure of $M$ and $N$. The goal is to find an embedding $g:M\hookrightarrow N$ over $A$.

\emph{Step 1. We may assume $\VF(A)$ is a field.}

Indeed, by Lemma~\ref{L:lambda-subfield}, if $\VF(A)$ is not a field, then $\VF(A)\subseteq \VF(M)^{p^\infty}$ in which case $\text{Frac}(\VF(A))\subseteq \VF(M)^{p^\infty}$ as well, so for $a,b_1,\ldots,b_n\in\text{Frac}(\VF(A))$, we have
	\[
	\lambda_{n,\nu}(a;b_1,\cdots,b_n)=
	\begin{cases}
		a^{1/p},   & n=0; \\
		0,             & \text{otherwise.}
	\end{cases}
	\]
It follows that $A':=(\text{Frac}(\VF(A)),\VG(\cA),\RF(\cA))$ is an $\cLgk^{\lambda}$-substructure, and its isomorphism type is uniquely determined over $A$.

\emph{Step 2. We may assume $\RF(A)=\RF(M)$ and $\VG(A)=\VG(M)$.}

This follows from quantifier elimination for $\ACF_p$ in $\cLr$ and for $\DOAG$ in $\cLog$.

The following is the main step:

\begin{claim}
    Let $A\subseteq M_0\subseteq M$, where $M_0$ is an $\cLgk$-substructure of $M$ such that $\VF(M_0)$ is a field and the field extension $\VF(M_0)/\VF(A)$ is finitely generated. 
    
    Then there is an $\cLgk$-embedding $f:M_0\hookrightarrow N$ over $A$ such that whenever a tuple $\mathbf{b}=(b_1,\ldots,b_d)$ from $\VF(M_0)$ is $p$-independent over $\VF(A)$ in $\VF(M)$, then $f(\mathbf{b})$ is $p$-independent over $\VF(A)$ in $\VF(N)$.
\end{claim}

Observe that $e(\VF(N)/\VF(A))\geq e(\VF(M)/\VF(A))$ by the saturation assumption on $N$. 

To prove the claim, assume first that $\VF(M_0)=\VF(A)(t)$ for some element $t$ which is $p$-independent in $\VF(M)$ over $\VF(A)$. Then necessarily $t$ is transcendental over $\VF(A)$. Thus, by Fact~\ref{F:QE-ACVF}, every quantifier-free $\cLgk$(A)-formula $\phi(x)$ satisfied by $t$ has non-empty interior. 

Since $e(\VF(N)/\VF(A))\geq e(\VF(M)/\VF(A))$, there is $t_0\in\VF(N)$ which is $p$-independent over $\VF(A)$. As $\VF(N)^p$ is dense in $\VF(N)^{alg}$ (by Fact~\ref{F:dense}), so is $t_0\VF(N)^p$. By saturation, it follows that $\qftp_{\cLgk}(t/A)$ is realized in $\VF(N)$ by an element $t'$ which is $p$-independent over $\VF(A)$ in $\VF(N)$.

Now let $M_0$ be arbitrary as in the claim. By Fact~\ref{F:characterize-sep-ext}, the finitely generated separable extension $\VF(M_0)/\VF(A)$ admits a transcendence basis $(b_1,\ldots,b_n)$ such that the extension $\VF(M_0)/\VF(A))(b_1,\ldots,b_n)$ is separable algebraic. We may assume that $(b_1,\ldots,b_d)$  form a $p$-basis of $\VF(M_0)$ over $\VF(A)$ in $\VF(M)$. By the special case, treated at the beginning, and induction on $d$, we find a tuple $(b_1',\ldots, b_d')$ in $\VF(N)$ which is $p$-independent over $\VF(A)$ such that $b_i\mapsto b_i'$ defines an $\cLgk$-embedding $f_0:A(b_1,\ldots,b_d)\hookrightarrow N$ over $A$. By the density of $\VF(N)$ in $\VF(N)^{alg}$, using Fact~\ref{F:QE-ACVF} again, $f_0$ extends to an $\cLgk$-embedding $f_1: A(b_1,\ldots,b_n)\hookrightarrow N$, which extends to an $\cLgk$-embedding $f:M_0\hookrightarrow N$, as $\VF(M_0)/\VF(A)(b_1,\ldots,b_n)$ is separable algebraic and $\VF(N)$ is separably closed.

It is straight forward to check that $f$ is as required, finishing the proof of the claim.

By the claim and the saturation assumption on $N$, the partial type
$$\qftp_{\cLgk}(M/A) \cup \left\{ 
\begin{aligned}
\text{``$(x_{b_1}, \ldots, x_{b_m})$ is $p$-independent over $\VF(A)$ in $\VF(N)$''} \mid \,\,\,\,\,\,\,\,\,\,\,\,\,\,\,\,\\
\,\,\,\,\,\,\,\,\,\,\,\,\,\,\,\,\text{$(b_1, \ldots, b_m)$ is $p$-independent in $\VF(M)$ over $\VF(A)$}
\end{aligned} \right\}$$
is realized in $N$. This yields an $\cLgk$-embedding $g:M\hookrightarrow N$. Note that, by construction, $\VF(N)/g(\VF(M))$ is separable. Thus $g$ preserves the functions $\lambda_{n,\nu}$ as well and is thus an $\cLgk^\lambda$-embedding as required.

Completeness of $\SCVF_{p,e}$ follows from quantifier elimination since $\mathbb{F}_p$ is a substructure of every model.
\end{proof}

\section{A quick review of beautiful pairs}\label{S:beauty}
In this section we recall the key notions and results from \cite{CuHiYe26} on beautiful pairs in not necessarily stable theories.

We fix a language $\cL$ and a complete $\cL$-theory $T$ with quantifier elimination. We let $\cU$ be a monster model of~$T$. 
A small subset of $\cU$ is a subset of cardinality smaller than the degree of saturation of $\cU$. The parameter sets we consider are small subsets of $\cU$, and the models of $T$ we consider are small elementary substructures of $\cU$.

\subsection*{Definable types and stably embedded pairs}
Recall that given $A\subseteq M\models T$, a type $p\in S_x(M)$ is \emph{$A$-definable} if for every $\cL$-formula $\varphi(x,y)$ there is an $\cL(A)$-formula $d_p\varphi(y)$ such that for every $c\in M^y$
\begin{equation}\label{Eq:def-scheme}
\varphi(x,c)\in p(x) \Leftrightarrow \ M\models d_p\varphi(c). 
\end{equation}
We say $p\in S_x(M)$ is \emph{definable} if it is $M$-definable. Given any set $B$ containing $M$, we use $p|B$ to denote the type $\{\varphi(x,b)\mid b\in B^{y}\text{ such that } \cU\models d_p\varphi(b)\}$, i.e.\ the (complete) type given by applying the scheme of definition \eqref{Eq:def-scheme} to tuples from $B^{y}$. We let $S_x^\Def(A)$ denote the subset of $S_x(\cU)$ of $A$-definable types, and for an $A$-definable set $X$, we denote by $S_X^\Def(A)$ the set of global $A$-definable types concentrating on $X$. Here, by a 'global' type we mean a type over $\cU$.

Stably embedded (and in particular beautiful) pairs are a way to turn definable types into structures, in order to facilitate their model-theoretic study. Let $\cL_P:=\cL\cup\{P\}$ be the language of pairs of $\cL$-structures. An $\cL_P$-structure $\cA=(A,P(\cA))$ is called a \emph{stably embedded pair}, in short, \emph{\sepa-pair}, if $P(\cA)\models T$, $P(\cA)\subseteq A\models T_{\forall}$ and $\tp_{\cL}(A/P(\cA))$ is definable. We denote by $\cK_\Def$ the class of \sepa-pairs. 

Given a model $A\models T$, one associates to $\cA$ the trivial \sepa-pair $A_\triv:=(A,A)$. If $\cA,\cB\in\cK_{\Def}$ we write $\cA\subseteq_{\BP}\cB$ and call $\cA$ a \emph{$\BP$-substructure} of $\cB$, if $\cA\subseteq_{\cL_P}\cB$ and $\tp_{\cL}(A/P(\cB))=\tp_{\cL}(A/P(\cA))|P(\cB)$. A \emph{$\BP$-embedding} is an isomorphism onto a $\BP$-substructure.

Given $\cA\in\cK_{\Def}$ and $B$ small such that $P(\cA)\elres B\elres\cU$, it follows from the definitions that there is a unique amalgam $\cA_B$ of \sepa-pairs with $\BP$-embeddings
\[
\begin{tikzcd}
             & \cA_B                       &              \\
\cA \arrow[ru,dotted,"g_1"] &                     & B_\triv \arrow[lu,dotted,"g_2"'] \\
             & P(\cA)_\triv \arrow[lu] \arrow[ru] &             
\end{tikzcd}
\]
such that $\mathcal{A}_B=(\langle g_1(A)\cup g_2(B)\rangle_{\cL},B)$, where $\langle C\rangle$ denotes the $\cL$-substructure generated by $C$. 

Let $\cF$ denote a class of global definable types. For $A\subseteq \cU$ and a tuple of variables $x$ (possibly infinite), we let $\cF_x(A)$ denote the subset of $\cF$ of $A$-definable types in variables $x$. If $X$ is an $A$-definable set, $\cF_X(A)$ denotes the subset of $\cF(A)$ concentrating on $X$.

\begin{definition}\label{def:natuclass-F} A non-empty class $\cF$ of global definable types is called \emph{natural} if $\cF$ satisfies the following properties for every $\cM\preccurlyeq \cU$:
\begin{itemize}[leftmargin=*]
    \item For any tuple $x$ of variables, $p\in \cF_x(M)$ if and only if $p|_{x'}\in \cF_{x'}(M)$ for any finite $x'\subseteq x$.
    \item Given $\sigma\in \Aut(\cU)$, $\sigma(\cF(M))=\cF(\sigma(M))$. 
    \item $\cF(M)$ is closed under push forwards by $M$-definable functions.
\end{itemize}
\end{definition}

Given a natural class of global definable types $\cF$, one defines the class of \sepa-pairs $\cK_\cF$ associated to $\cF$ which consists of the \sepa-pairs $\cA$ which satisfy $\tp_{\cL}(A/P(\cA))\in\cF(P(\cA))$. One may characterize axiomatically the subclasses $\cK$ of $\cK_{\Def}$ which are of the form $\cK_{\cF}$ for some natural class of global definable types $\cF$ (see \cite[Definition~2.26 \& Lemma~2.3.3]{CuHiYe26}). We call them \emph{natural classes} of \sepa-pairs. Given a natural class $\cK\subseteq\cK_{\Def}$, one denotes by $\cF_{\cK}$ the (uniquely determined) natural class of definable types with $\cK=\cK_{\cF}$. 

The two extreme natural classes of definable types are the class $S^{\Def}$ of all definable types and the class $\cF^{\triv}$ consisting of all realized types, with corresponding natural classes $\cK_{\Def}$ and $\cK_\triv$ of \sepa-pairs. The class of all definable types which are orthogonal to a given sort (in the sense of \cite[Section 2.5]{HL}), or to a given set of $\emptyset$-invariant types, is natural. Another natural class is that of all generically stable types. 

\subsection*{Beautiful pairs, beauty transfer and its consequences} 
We start with the definition of the key concept. 

\begin{definition}\label{def:bpair} Let $\cK$ be a natural class of \sepa-pairs and $\lambda\geqslant |T|^+$. An \sepa-pair $\cM\in\cK$ is called a \emph{$\lambda$-$\cK$-beautiful pair} if the following condition is satisfied:
\begin{itemize}\label{BP-condition}
    \item[(BP)] whenever there are $\BP$-embeddings $f\colon \cA \to \cM$ and $g\colon \cA \to \cB$ with $\cA,\cB\in \cK$ and $|\cB|<\lambda$, there is a $\BP$-embedding $h\colon \cB\to \cM$ such that $f=h\circ g$. 
\end{itemize}
We call an $\cL_P$-structure $\cM$ a \emph{$\cK$-beautiful} pair if it is a $\lambda$-$\cK$-beautiful pair for some $\lambda\geqslant|T|^+$, and we will omit the $\cK$ when $\cK=\cK_{\Def}$.
\end{definition}



The natural class $\cK$ is said to have 
\begin{itemize} 
\item the \emph{amalgamation property} (AP) if whenever $f\colon\cA\to\cB$ and $g\colon\cA\to\cC$ are $\BP$-embeddings between structures from $\cK$, there are $\cD\in\cK$ and $\BP$-embeddings $h\colon\cB\to\cD$ and $i\colon\cC\to\cD$ such that $h\circ f=i\circ g$; 
\item the \emph{extension property} (EP) if whenever $x$ is a single variable, $\cA\in\cK$ and $\varphi(x)$ a consistent (with $T$) $\cL(A)$-formula, then there is $\cB\in \cK$, a $\BP$-embedding $f\colon\cA\to \cB$ and $b\in B$ satisfying $\varphi(x)$. 
\end{itemize}

\begin{fact}[\mbox{\cite[Theorem~2.4.5, Corollary~2.4.6 \& Lemma~2.4.9]{CuHiYe26}}]\label{fact:charct-beaut-pairs} For a natural class $\cK$, the following holds.
\begin{enumerate}
    \item $\cK$ has AP if and only if $\lambda$-$\cK$-beautiful pairs exist for all $\lambda\geqslant |T|^+$. Moreover, assuming $\cK$ has AP, $\cK$-beautiful pairs are elementary pairs if and only if $\cK$ has EP.
    \item Any two $\cK$-beautiful pairs are back-and-forth equivalent, so in particular elementary equivalent. If $\cK$ has AP, the common theory of all $\cK$-beautiful pairs, denoted by $T_\BP(\cK)$, is thus (consistent and) complete.
    \item Assume AP holds in $\cK$. Then the following are equivalent.
    \begin{enumerate}
      \item  There exists $\lambda\geqslant |T|^+$ and $\cM\models T_\BP(\cK)$ which is $\lambda$-saturated and $\lambda$-$\cK$-beautiful; 
    \item  for every $\lambda\geqslant |T|^+$ and every $\cM\models T_\BP(\cK)$, $\cM$ is $\lambda$-saturated if and only if $\cM$ is $\lambda$-$\cK$-beautiful. 
    \end{enumerate}
    \end{enumerate}
\end{fact}

When $\cK=\cK_{\Def}$, we will simply write $T_\BP$ instead of $T_\BP(\cK)$.

\begin{definition}Let $\cK$-beautiful be a natural class of \sepa-pairs with AP. We say that $\cK$ has \emph{beauty transfer} if one of the equivalent conditions in Fact~\ref{fact:charct-beaut-pairs} (3) holds. 
\end{definition}

Given a natural class $\cF$ of definable types, we say $T$ has UDDT($\cF$) if for every $\cL$-formula $\varphi(x, y)$ there is an $\cL$-formula $\psi(y,z)$ such that for every model $M$ of $T$ and every $p\in \cF_x(M)$ there is a $c\in M^z$ such that $\psi(y,c)\equiv d_p\varphi(y)$.

Part (2) of the following result was the main motivation for the introduction and study of beautiful pairs in \cite{CuHiYe26}.

\begin{fact}[\mbox{\cite[Proposition~2.4.12, Theorem~2.5.4 \& Theorem~2.4.16]{CuHiYe26}}]\label{fact:main-beauty}
    Let $\cF$ be a natural class of definable types. Assume that $\cK_{\cF}$ has AP and beauty transfer. Then the following holds.
    \begin{enumerate}
        \item $T$ has UDDT($\cF$).
        \item If $X$ is an $A$-definable set, then $\cF_X(\cU)$ may be canonically identified with a strict pro-definable set over $A$, i.e., with the projective limit of $A$-definable sets, where the projective system has a small index set and surjective $A$-definable transition functions. 
        \item $\cL_P$-isomorphisms between $\BP$-substructures of models of $T_{\BP}(\cK)$ are elementary maps, and the predicate $P$ is purely stably embedded in models of $T_{\BP}(\cK)$, i.e., it is stably embedded with induced structure given by $\cL$.
    \end{enumerate}
\end{fact}

\subsection*{Examples of theories of beautiful pairs}

Assuming that $T$ has UDDT($\cF^{\Def}$), it is easy to see that the class of stably embedded pairs of models of $T$ is elementary (see, e.g., \cite{cubi-ye}). We let $T_{\SE}$ be the $\cL_P$-theory of this class.

All the examples we will discuss in what follows will have UDDT($\cF^{\Def}$). Let us start with a remark, illustrating that things may be more complicated than one might expect.

\begin{remark}
\begin{enumerate}
    \item In \cite[p.\ 14]{CuHiYe26} a theory $T$ is constructed in which $\cK_{\Def}$ has AP and beauty transfer but not EP. 
    \item Building on the example from (1), Mennuni and the second author constructed a theory in which $\cK_{\Def}$ does not have AP \cite{hils-mennuni2}.
\end{enumerate}
\end{remark}

The following is due to Poizat \cite{poizat} (see \cite[Section~3.1]{CuHiYe26} for a discussion).

\begin{fact}\label{fact:stable-nfcp}
Assume $T$ is stable. Then $\cK_{\Def}$ has AP and EP. Moreover, $\cK_{\Def}$ has beauty transfer if and only if $T$ is nfcp.    
\end{fact}

Observe that, trivially, $\cK_\triv$ has AP, EP and beauty transfer in any theory $T$.

\begin{example}[\mbox{$\ACF$ \cite[Proposition~3.1.3]{CuHiYe26}}]\label{Ex:ACF}
Let $T$ be a completion of the $\cLr$-theory $\ACF$ of algebraically closed fields. Then $T_{\SE}$ has two completions, both corresponding to $T_\BP(\cK)$ for a natural class of \sepa-pairs with AP, EP and beauty transfer. 

The theory of trivial pairs corresponds to $T_\BP(\cK_\triv)$, and the theory of non-trivial pairs to $T_\BP=T_\BP(\cK_{\Def})$.
\end{example}

Let $\cLog=\{0,+,<\}$ be the language of ordered groups, and let $\DOAG$ be the $\cLog$-theory of non-trivial divisible ordered abelian groups. It is well-known that $\DOAG$ is complete with quantifier elimination, so in particular $o$-minimal. Thus, by the Marker-Steinhorn Theorem \cite{MS}, $\DOAG_{\SE}$ equals the theory of pairs of models $(M,P(M))$ of $\DOAG$ such that $P(M)$ is Dedekind complete in $M$. 

Let $\cF^{\bdd}$ be the natural class of all bounded definable types, i.e., definable types orthogonal to the type $p_{+\infty}$ at infinity, and let $\cF^{\OOnv}$ be the natural class of all definable types orthogonal to the type $p_{0^+}$ of an infinitesimal element. We denote by $\cK_{\bdd}$ and $\cK_{\OOnv}$ the corresponding natural classes of \sepa-pairs in $\DOAG$. Consider the following axioms in the language of pairs:

\begin{align*}
\label{infty}\tag{$E_{\infty}$} & \exists x\forall y(P(y) \to y<x);  \\
\label{0+}\tag{$E_{0^+}$}  & \exists x\forall y((P(y) \wedge 0<y) \to 0<x<y).   
\end{align*}

\begin{example}[\mbox{$\DOAG$ \cite[Theorem~4.1.1]{CuHiYe26}}]\label{Ex:DOAG}
For $T=\DOAG$, the theory $T_{\SE}$ has four completions, corresponding to $T_\BP(\cK)$ for $\cK$ equal to $\cK_{\triv}$, $\cK_{\bdd}$, $\cK_{\OOnv}$ and $\cK_{\Def}$. All these natural classes of \sepa-pairs have AP, EP and beauty transfer. Moreover,
\begin{enumerate}
    \item $T_\BP(\cK_\triv)$ is axiomatized by $T_{\SE}\cup\{\neg E_{\infty},\neg E_{0^+}\}$.
    \item $T_\BP(\cK_\bdd)$ is axiomatized by $T_{\SE}\cup\{\neg E_{\infty},E_{0^+}\}$.
    \item  $T_\BP(\cK_\OOnv)$ is axiomatized by $T_{\SE}\cup\{E_{\infty},\neg E_{0^+}\}$.
    \item  $T_\BP$ is axiomatized by $T_{\SE}\cup\{E_{\infty},E_{0^+}\}$.
\end{enumerate}
\end{example}

Next we will discuss beautiful pairs for $\ACVF$. Recall that if $K\subseteq L$ is an extension of valued fields and $(a_,\ldots,a_n)\in L^n$, the tuple $(a_,\ldots,a_n)$ is called \emph{$K$-valuation independent} if $v(\sum_{i=1}^n c_ia_i) = \min_{i}(v(c_ia_i))$ for every $K$-linear combination $\sum_{i=1}^n c_ia_i$. The extension $L/K$ is called \emph{$vs$-defectless}\label{VS-def}\footnote{This is the same as ``separated'' in W. Baur's and F. Delon's terminology.} if every finite-dimensionl $K$-vector subspace $V$ of $L$ admits a \emph{$K$-valuation basis}, that is, a $K$-valuation independent tuple which spans $V$ over $K$. 

Given a completion $T$ of $\ACVF$ in the 3-sorted language $\cL=\cLgk$, an $\cL_P$-structure $\cA=(A,P(\cA))$ is an \sepa-pair if $P(\cA)\elres\cU\models T$, and $A$ is a substructure of $\cU$ such that $\VF(\mathrm{Frac}(A))/P(\cA)$ is $vs$-defectless, $(\VG(A),\VG(P(\cA)))$ is an \sepa-pair in the sense of $\DOAG$ and $(\RF(A),\RF(P(\cA)))$ is an \sepa-pair in the sense of $\ACF_p$, where $p$ is the characteristic of $\RF(\cU)$. It follows that $T$ has UDDT($\cF^{\Def}$), and $T_{\SE}$ is given by the theory of $vs$-defectless elementary pairs of models of $T$ such that both the value group and the residue field pair are \sepa-pairs.

Now let $T_{\RF}$ be a completion of $(\ACF_p)_{\SE}$ and $\cK_{\RF}$ the natural class with $T_{\RF}=(\ACF_p)_\BP(\cK_{\RF})$ (see Example~\ref{Ex:ACF}). Similarly, let $T_{\VG}$ be a completion of $\DOAG_{\SE}$ and $\cK_{\VG}$ the natural class with $T_{\VG}=\DOAG_\BP(\cK_{\VG})$ (see Example~\ref{Ex:DOAG}). Then the class $\cK(\cK_{\RF},\cK_{\VG})$ of \sepa-pairs $\cA$ in $T$ such that $(\RF(A),\RF(P(\cA)))\in\cK_{\RF}$ and $(\VG(A),\VG(P(\cA)))\in\cK_{\VG}$ is a natural class. Let $T(T_{\RF},T_{\VG})$ be $T_{\SE}$ together with axioms expressing that the residue field pair is a model of $T_{\RF}$ and the value group pair is a model of $T_{\VG}$.

\begin{example}[\mbox{$\ACVF$ \cite[Theorem~7.1.5 \& Corollary~7.1.6]{CuHiYe26}}]\label{Ex:ACVF}
For $T$ a completion of $\ACVF$, the theory $T_{\SE}$ has eight completions, corresponding to $T_\BP(\cK)$ for $\cK$ equal to $\cK(\cK_{\RF},\cK_{\VG})$ for some $\cK_{\RF}\in\{\cK_{\RF}^\triv,\cK_{\RF}^{\Def}\}$ and $\cK_{\VG}\in\{\cK_{\VG}^\triv,\cK_{\VG}^\bdd,\cK_{\VG}^\OOnv,\cK_{\VG}^\Def\}$. All these natural classes of \sepa-pairs have AP, EP and beauty transfer. Moreover, $T_\BP(\cK(\cK_{\RF},\cK_{\VG})))=T(T_{\RF},T_{\VG})$.
\end{example}

\section{An analogue of Keisler's theorem for separable extensions}\label{S:Keisler}
A theorem of Keisler's from \cite{Kei64} asserts that the theory of pairs of fields $(K,P(K))$, assuming $K\models\ACF$ and the extension $K/P(K)$ is infinite, is determined by the $\cLr$-theory $T$ of $P(K)$. In their recent work \cite{DeKaNe24}, d'Elb\'ee, Kaplan and Neuhauser have undertaken a comprehensive study of the theory of $(K,P(K))$ which they denote by $\ACF_T$, establishing in particular transfer results for model-theoretic tameness properties. In this section, we prove analogs of some of these results in the separable category.

\begin{notation}
    \begin{itemize}
    \item Given fields $K\subseteq L,M\subseteq\cU$ of characteristic $p>0$, we write $L\ind_K^{\ld,p}M$ as an abbreviation for $[L\ind_K^{p}M \text{ and }L\ind_K^{\ld}M]$.
        \item We denote by $\cL:=\cLr$ the language of rings, and we set $\cL_P:=\cL\dot{\cup}\{P\}$.
    \end{itemize}
\end{notation}

\begin{definition}
    Let $p$ be a prime number, $T$ a complete $\cL$-theory of fields of characteristic $p$, and let $e\in\N\cup\{\infty\}$. \begin{enumerate}
        \item We define $\SCF_T$ to be the $\cL_P$-theory expressing the following for an $\cL_P$-structure $(K,P(K))$:
    \begin{enumerate}
        \item $(K,P(K))$ is a pair of fields with $K/P(K)$ a separable extension,
        \item $P(K)\models T$,
        \item $K$ is separably closed, and
        \item $[K:P(K)]=\infty$.
            \end{enumerate}
            \item We define $\SCF_{T,e}$ to be the theory $\SCF_T$ together with the axiom  $e(K/P(K))=e$.
        \item Let $(K,P(K))\models \SCF_T$. We call a substructure $(K_0,P(K_0))\subseteq (K,P(K))$ \emph{good} if the following properties hold:
        \begin{enumerate}[(i)]
            \item $K_0$ is a subfield of $K$,
            \item $K/K_0$ is separable,
            \item $P(K)/P(K_0)$ is separable, and
            \item $K_0\ind^{\ld,p}_{P(K_0)}P(K)$.
            \end{enumerate}
            \item An isomorphism between good substructures $(K_0,P(K_0))$ and $(K_1,P(K_1))$ of models of $\SCF_T$ is called \emph{good} if its restriction to $P(K_0)$ is elementary in the sense of $T$.
    \end{enumerate}
  
    \end{definition}

    \begin{remark}
        If $(K_0,P(K_0))\subseteq (K,P(K))$ is a good substructure of a model of $\SCF_T$, then $K_0/P(K_0)$ is separable, as both $K/P(K)$ and $P(K)/P(K_0)$ are separable by assumption.
    \end{remark}

For $M_0\subseteq M_1,M_2\subseteq \cU\models \SCF_T$, where $M_i$ is a good substructure of $\cU$ for $i=0,1,2$ we define 
\[M_1\ind^*_{M_0}M_2\,\,\,\,\, :\Leftrightarrow \,\,\,\,\,\left[M_1
\ind^{\ld,p}_{M_0P(\cU)}M_2\text{ and }P(M_1)\ind^{\ld,p}_{P(M_0)} P(M_2)\right]\]




\begin{lemma}\label{L:Abstract-ind}
Assume that $A_0\subseteq A_i\subseteq\cU=(U,P(U))$ for $i=1,2$ and  that $\ind^+$ is a symmetric ($A\ind^+_BC$ iff $C\ind^+_{B}A$) and fully transitive (for $C\subseteq D$ one has $A\ind^+_BD$ iff $[A\ind^+_BC$ and $A\ind^+_{BC}D]$) ternary relation on subsets of $U$. Assume that $A_i\ind^+_{P(A_i)}P(U)$ for $i=0,1,2$, $A_1\ind^+_{A_0P(U)}A_2$ and $P(A_1)\ind^+_{P(A_0)}P(A_2)$. Then the following properties hold:
\begin{enumerate}[(i)]
\item $A_1\ind^+_{A_0}A_2$\label{item:abstr ind ohne P}
    \item $A_1A_2\ind^+_{P(A_1)P(A_2)}P(U)$\label{item:abstract product good}
\end{enumerate}
\end{lemma}

\begin{proof}
The proof is straight forward. We give it for convenience of the reader. 

By $A_1\ind^+_{P(A_1)}P(U)$, we get $A_1\ind^+_{A_0P(A_1)}P(U)$, which together with $A_1\ind^+_{A_0P(U)}A_2$ yields 
\begin{equation}\label{Eq:first}A_1\ind^+_{A_0P(A_1)}A_2P(U)
\end{equation}

Next, by $P(U)\ind^+_{P(A_2)}A_2$ we get $P(A_1)\ind^+_{P(A_2)}A_2$, which together with $P(A_1)\ind^+_{P(A_0)}P(A_2)$ yields 
\begin{equation}\label{Eq:second}P(A_1)\ind^+_{P(A_0)}A_2
\end{equation}

By \eqref{Eq:first} we have $A_1\ind^+_{A_0P(A_1)}A_2$, and from \eqref{Eq:second} we infer $P(A_1)\ind^+_{A_0}A_2$. Together, this yields $A_1\ind^+_{A_0}A_2$, proving \ref{item:abstr ind ohne P}.

Part \ref{item:abstract product good} is shown as follows. Part \ref{item:abstr ind ohne P} yields $A_1\ind^+_{A_2P(A_1)}P(U)$, and from $A_2\ind^+_{P(A_2)}P(U)$ we infer $A_2\ind^+_{P(A_2)P(A_1)}P(U)$. Together, these imply \ref{item:abstract product good}.
\end{proof}

\begin{lemma}\label{L:product}
Let $M_0\subseteq M_1,M_2\subseteq \cU\models \SCF_T$, where $M_i$ is a good substructure of $\cU$ for $i=0,1,2$. Assume that $M_1\ind^*_{M_0}M_2$. Then the following properties hold:
\begin{enumerate}[(a)]
    \item $M_1\ind^{\ld,p}_{M_0}M_2$.\label{item:ind-wo-P}
    \item $M_1M_2$ is a good substructure of $\cU$.\label{item:product-good}
\end{enumerate}
\end{lemma}

\begin{proof}
    Part \ref{item:ind-wo-P} follows from Lemma~\ref{L:Abstract-ind}\ref{item:abstr ind ohne P} as $\ind^{\ld}$ and$\ind^p$ are both symmetric and fully transitive. 
    
    As for part \ref{item:product-good}, separability of $\cU/M_1M_2$ follows from the separability of $\cU/M_i$ for $i=0,1,2$ together with $M_1\ind^p_{M_0}M_2$. Moreover, from Lemma~\ref{L:Abstract-ind}\ref{item:abstract product good} we infer $P(M_1M_2)=P(M_1)P(M_2)$. Separability of the extension $P(\cU)/P(M_1M_2)$ thus follows in a similar way, as $P(M_1)\ind^p_{P(M_0)}P(M_2)$. We leave the straight forward verification of the remaining property $M_1M_2\ind^{\ld,p}_{P(M_1M_2)}P(\cU)$ to the reader.
\end{proof}

    \begin{lemma}\label{L:Keisler-main}
    Let $\cK=(K,P(K))\models\SCF_{T,e}$ and $\cK'=(K',P(K'))\models \SCF_{T,e'}$, with $e\leq e'$, $\cK$ countable and $\cK'$ $\aleph_1$-saturated. Let $(K_0,P(K_0))$ be a good substructure of $\cK$ and $f:(K_0,P(K_0))\hookrightarrow\cK'$ a good embedding onto a good substructure of $\cK'$. Then $f$ extends to a good embedding $\tilde{f}:\cK\hookrightarrow\cK'$ with $\tilde{f}(\cK)$ a good substructure of $\cK'$.
    \end{lemma}

    \begin{proof}
    Let $f:(K_0,P(K_0))\hookrightarrow\cK'$ be a good embedding onto a good substructure of $\cK'$. Note that $(P(K_0),P(K_0))$ and $(P(K),P(K))$ are good substructures of $\cK$, and clearly we have $K_0\ind^*_{P(K_0)}P(K)$. Thus the base change $(K_0P(K),P(K))$ is a good substructure of $\cK$ by Lemma~\ref{L:product}. Moreover, by $\aleph_1$-saturation of $P(K')$, the partial elementary map (in the sense of $T$) $f\upharpoonright_{P(K_0)}:P(K_0)\hookrightarrow P(K')$ extends to an elementary embedding $f_P:P(K)\hookrightarrow P(K')$. Using Lemma~\ref{L:product}\ref{item:ind-wo-P}, $f\cup f_P$ induces a good embedding of $K_0P(K)$ onto a good substructure of $\cK'$. We may thus assume that $P(K_0)=P(K)$.

        \begin{claim}\label{Cl:1}
            Let $\cK=(K,P(K))$ be a pair of fields of characteristic $p$. For every $m\in\N$ there is a partial $\cL_P$-type $\pi_{
            K}(x_1,\ldots,x_m)$ over $K$ expressing, uniformly in every $\cL_P$-extension $(L,P(L))$ of pairs of fields of $\cK$,  that $(x_1,\ldots,x_m)$ is $p$-independent over $KP(L)$.
        \end{claim}

        To prove the claim, for $d\geq1$ set $X_d:=\{\sum_{i=1}^d a_ib_i\mid a_i\in K, b_i\in P(L)\}$. Then $\pi_{K}$ is given by the set of sentences negating, for any $d\geq1$, the existence of $0\neq g_\mu$ and $0\neq h_\mu$ from $X_d$ and $c_\mu\in L$ not all 0, for $\mu\in p^m$, such that
        $$\sum_\mu c_\mu^p \frac{g_\mu}{h_\mu}\overline{x}^{\mu}=0.$$
        This proves Claim~\ref{Cl:1}.

        Now we return to our proof. Assume that $f:(K_0,P(K_0))\hookrightarrow\cK'$ is a good embedding onto a good substructure of $\cK'$, where $(K_0,P(K_0))$ is a good substructure of $\cK$ and $P(K_0)=P(K)$.

          \begin{claim}\label{Cl:2}
            For any finite $A_0\subseteq K$ there is $A_0\subseteq B_0\subseteq K$ with $B_0$ finite and a field embedding $\tilde{f}:\widetilde{K_0}:=K_0(B_0)\hookrightarrow K'$ extending $f$ such that the following properties hold:
            \begin{enumerate}[(i)]
            \item $\tilde{f}(\widetilde{K_0})\ind^{ld}_{f(P(K_0))}P(K')$
            \item Whenever a tuple $\overline{a}$ from $A_0$ is $p$-independent in $K$ over $K_0=K_0P(K)$, then $\tilde{f}(\overline{a})$ is $p$-independent in $K'$ over $f(K_0)P(K')$.
            \end{enumerate}  
                 \end{claim}

To prove Claim~\ref{Cl:2}, choose $B_0\subseteq K$ finite containing $A_0$ such that, setting $\widetilde{K_0}:=K_0(B_0)$, the $p$-dimension of $A_0$ over $K_0$ is the same whether we compute it in $K$ or in $\widetilde{K_0}$. It then follows that for any $A_1\subseteq A_0$ and $a\in A_0$, $a\in\cl_{\widetilde{K_0}}(K_0A_1)$ if and only if $a\in \cl_K(K_0A_1)$. Let $(a_1,\ldots,a_d)$ be a $p$-basis of $A_0$ over $K_0$ (in $\widetilde{K_0}$). Choose $a_{d+1},\ldots,a_h\in \widetilde{K_0}$ such that $\overline{a}:=(a_1,\ldots,a_h)$ is a $p$-basis of $\widetilde{K_0}/K_0$. By \cite[Theorem~15]{Mac39}, $\overline{a}$ is algebraically independent over $K_0$ and $\widetilde{K_0}/K_0(\overline{a})$ is separable algebraic.   

By $\aleph_1$-saturation of $\cK'$ and as $e\leq e'$ by assumption, there is a tuple $(a_1',\ldots,a_h')$ in $K'$ which is algebraically independent over $f(K_0)P(K')$ and such that $(a_1',\ldots,a_d')$ is $p$-independent over $f(K_0)P(K')$. As $\widetilde{K_0}/K_0(\overline{a})$ is separable algebraic and $K'$ is separably closed, $\overline{a}\mapsto\overline{a'}$ extends to a field embedding $\tilde{f}:\widetilde{K_0}\hookrightarrow K'$ extending $f$ and satisfying $\tilde{f}(\widetilde{K_0})\ind^{alg}_{f(P(K_0))}P(K')$. Thus (i) follows, as the extension $P(K')/f(P(K_0))$ is regular (being elementary). Property (ii) holds by construction, proving Claim~\ref{Cl:2}.

The two claims, combined with $\aleph_1$- saturation of $\cK'$, yield the existence of good embedding $\tilde{f}:\cK\hookrightarrow\cK'$ extending $f$. 
        \end{proof}

    \begin{theorem}\label{T:Keisler}
        For any complete theory $T$ of fields of characteristic $p>0$ we have the following.
        \begin{enumerate}
            \item \label{part:complete}The completions of $\SCF_T$ are given by $\SCF_{T,e}$, for $e\in\N\cup\{\infty\}$.
        \item \label{part:good}For a fixed $e\in\N\cup\{\infty\}$, any good isomorphism between good substructures of models of $\SCF_{T,e}$ is an elementary map
        \item \label{part:SE-Keisler}In $\SCF_{T,e}$, $P$ is stably embedded with induced structure a pure field.
        \end{enumerate}
    \end{theorem}

    \begin{proof}
      Part (\ref{part:good}) follows directly from Lemma~\ref{L:Keisler-main}. As for part (\ref{part:complete}), this is a consequence of (\ref{part:good}), as $\Fp$ is a good substructure of any model of $\SCF_{T,e}$. 

      Finally, let us prove part (\ref{part:SE-Keisler}). First, let $K\elres L\models \SCF_{T,e}$, and let $P(K)\elres A_i\elres P(L)$ for $i=1,2$ with $\tp_T(A_1/P(K))=\tp_T(A_2(P(K))$. Denote by $\sigma$ the corresponding field isomorphism between $A_1$ and $A_2$ over $K$. Then the subfields $B_i=KA_i$, for $i=1,2$, are good substructures of $L$ (see the proof of Lemma~\ref{L:Keisler-main}) and $\id_K\cup\sigma:B_1\cong B_2$ is a good isomorphism, where $\sigma:A_1\cong A_2$ is the isomorphism over $P(K)$. Thus, by (\ref{part:good}), we infer $\tp_{\cL_P}(A_1/K)=tp_{\cL_P}(A_2/K)$, proving stable embeddedness of $P$, with induced structure given by $\cL$ (maybe over parameters).

      As for purity, let $A_i\elres P(K)$ for $i=1,2$ such that $\tp_T(A_1)=\tp_
      T(A_2)$, say witnessed by the field isomorphism $\sigma:A_1\cong A_2$. Then $\sigma$ is a good isomorphism, thus an elementary map in the sense of $\SCF_{T,e}$ by (\ref{part:good}). The result follows.
    \end{proof}

    \begin{remark}
        Actually, as in the analogous case of $\ACF_T$ treated in \cite[Section~4]{DeKaNe24}, the results in Theorem~\ref{T:Keisler} hold \emph{resplendently}, i.e., for any expansion of the small field $P(K)$, working with $T$ equal to the theory of $P(K)$ in the expanded language.
    \end{remark}

\begin{theorem}\label{T:preserve-tame}
Assume that $T$ is stable. Then $\SCF_T$ is stable. The same holds for `NIP' in place of `stable'.
\end{theorem}

\begin{proof}
    Assume first that $T$ is stable. Let $K\elres \cU\models \SCF_{T}$ with $|K|=\kappa$. 

    \begin{claim*}
        For any $a\in\cU$ there are $K'\elres L'\elres\cU$ with $L'$ countable, $K'\subseteq K$ and $a\in L'$ such that
        $L'\ind^*_{K'P(\cU)}K$ and $P(L')\ind^T_{P(K')}P(K)$.
    \end{claim*}
We construct two chains $(K_i)_{i<\omega}$ and $(L_i)_{i<\omega}$ of countable elementary substructures of $\cU$ such that $a\in L_0$ and for every $i<\omega$ we have $K_i\subseteq K$, $L_i\supseteq K_i$, $L_{3i}\ind^{\ld}_{K_{3i+1}P(\cU)}K$, $L_{3i+1}\ind^{p}_{K_{3i+2}P(\cU)}K$, and finally $P(L_{3i+2})\ind^{T}_{P(K_{3i+3})}P(K)$. 

In the inductive step, we find $K_{n+1}$ by local character of $\ind^{\ld}$, $\ind^p$ and $\ind^T$, respectively, and $L_{n+1}$ is chosen to be any countable elementary substructure of $\cU$ containing $L_nK_{n+1}$.

Setting $K':=\bigcup_{i}K_i$ and $L':=\bigcup_i L_i$ we get $L'\ind^{\ld,p}_{K'P(\cU)}K$ and $P(L')\ind^T_{P(K')}P(K)$.
Note that the latter property implies $P(L_{3i+2})\ind^{\ld,p}_{P(K_{3i+3})}P(K)$, so we infer $L'\ind^*_{K'P(\cU)}K$ as well, which proves the claim.

From Lemma~\ref{L:product} we infer that $L'\ind^*_{K'}K$. Now, $\qftp_{\cL_P}(L'K/K)$ is the unique extension of $\qftp_{\cL_P}(L'/K')$ satisfying  $L'\ind^*_{K'}K$. Moreover, as $T$ is stable, $\tp_{\cL}(P(L')P(K)/P(K))$ is uniquely determined by $\tp_{\cL}(P(L')/P(K'))$ together with $P(L')\ind^T_{P(K')}P(K)$. As $L'K$ is a good substructure of $\cU$ by Lemma~\ref{L:product} and as good isomorphisms are elementary maps, $\tp_{\cL_P}(L'/K)$ is determined by $\tp_{\cL_P}(L'/K')$ together with $L'\ind^*_{K'}K$ and $P(L')\ind^T_{P(K')}P(K)$, showing that 
$$|S_1(K)\leq2^{\aleph_0}\lambda^{\aleph_0},$$
thus finishing the proof for $T$ stable.

To prove the second part of the theorem, assume that $T$ is NIP. Recall that a theory $S$ is NIP if and only if for any $M\elres N\models S$ with $|M|=|N|$, $\tp(N/M)$ has $\leq 2^{|S|+|M|}$ many global coheirs (see, e.g., \cite[Proposition~2.43]{simon}). Now let $K'\elres \cU\models \SCF_T$ and $K'\elres L'$ with $|K'|=|L'|=\kappa$. If $\tp_{\cL_P}(L'/\cU)$ is a coheir of $\tp_{\cL_P}(L'/K')$, clearly 
$L'\ind^*_{K'}\cU$ and $P(L')\ind^{\ld,p}_{P(K')}P(\cU)$. The latter holds as $\tp_{\cL}(P(L')/P(\cU))$ is a coheir of $\tp_{\cL}(P(L')/P(K'))$. Since $T$ is NIP, there are at most $2^{\kappa}$ many global coheirs of $\tp_{\cL}(P(L')/P(K'))$. Arguing as in the proof of the stable case, we conclude that any coheir of $\tp_{\cL_P}(L'/K')$ is determined by the respective coheir of $\tp_{\cL}(P(L')/P(K'))$, proving that $\SCF_T$ is NIP.
\end{proof}

\begin{remark}
    The corresponding result for $\ACF_T$, in case $T$ is stable or NIP, is proven in \cite[Theorem~5.16 \& Theorem~5.34]{DeKaNe24}, resplendently with respect to the theory $T$. Similarly, our proofs of preservation of stability and NIP work resplendently, as may be checked in a straight forward manner.
\end{remark}

We suspect that various other tameness notions are preserved as well. In particular, it seems plausible that the methods from \cite{DeKaNe24} employed to show preservation of simplicity and NSOP$_1$ for $\ACF_T$ may be transferred to the case of $\SCF_T$. We did not investigate this further and leave it as a question.

\begin{question}
    Does the preservation result in Theorem~\ref{T:preserve-tame} hold for other tameness notions, in particular for simplicity, NTP$_2$, NIP$_n$ and/or NSOP$_n$ (where $n\in\N$)?
\end{question}

As $\SCF_{p,e}$ is complete and stable for every $e$, $\cK_\Def$ has AP and EP by Fact~\ref{fact:stable-nfcp}. We infer the following simple axiomatizations for $\left(\SCF_{p,e}\right)_{\BP}$, for finite and infinite $e$.

\begin{corollary}
    \begin{enumerate}
        \item For $e\in\N$, the theory of proper elementary pairs of models of $\SCF_{p,e}$ is given by $\SCF_{\left(\SCF_{p,e}\right),0}$. It is complete and equal to $\left(\SCF_{p,e}\right)_{\BP}$.
        \item In infinite degree of imperfection, one has $\left(\SCF_{p,\infty}\right)_{\BP}=\SCF_{\left(\SCF_{p,\infty}\right),\infty}$.
    \end{enumerate}
\end{corollary}

\begin{proof} Let $\cK=(K,P(K))$ be a beautiful pair of models of $\SCF_{p,e}$. Clearly, $\cK\models \SCF_{\left(\SCF_{p,e}\right),0}$ for $e$ finite and $\cK\models \SCF_{\left(\SCF_{p,\infty}\right),\infty}$ for $e=\infty$. The result now follows from Theorem~\ref{T:Keisler}.
\end{proof}

It is well known that $\SCF_{p,e}$ is nfcp for all $p,e$. We obtain a simple proof of this.

\begin{corollary}
For any $e$, $\SCF_{p,e}$ is nfcp. Equivalently, $\cK_{\Def}$ has beauty transfer in $\SCF_{p,e}$. Moreover, in $\SCF_{p,\infty}$, for any $e\in\mathbb{N}$ the class $\cF^{e}$ of global definable types whose realizations have imperfection degree $\leq e$ over $\cU$ is natural. The corresponding natural class $\cK_e$ has AP, EP and beauty transfer, and one has $\left(\SCF_{p,\infty}\right)_{\BP}(\cK_e)=\SCF_{\left(\SCF_{p,\infty}\right),e}$.
\end{corollary}

\begin{proof}
Beauty transfer for $\cK_{\Def}$, and also for $\cK_e$ in infinite degree of imperfection, follows from Lemma~\ref{L:Keisler-main}. The rest is clear.
\end{proof}

\begin{remark}
    Set $e':=e+e_T$, where $e_T\in\N\cup\{\infty\}$ is the imperfection degree of the models of $T$. Then $\SCF_{T,e}\models\SCF_{p,e'}$. \qed
\end{remark}

\section{Domination and definable types in separably closed valued fields}\label{S:Dom-SCVF}
In this subsection, we will generalize some results known for $\ACVF$ to separably closed valued fields, namely results about domination (in particular from \cite{CuHiYe26}) and the characterization of definable types (due to Cubides and Delon \cite{CD}). 


Recall that a valued field is called \emph{maximally complete} if it does not admit any proper immediate extension. We call it \emph{separably maximally complete}  if it does not admit any proper immediate separable extension. 


\begin{fact}[\mbox{\cite[Lemma~4.15]{HKR2016}}]\label{fact:char-sep-max}
For $K\models \SCVF_p$, the following are equivalent:
\begin{enumerate}
    \item $K$ is separably maximally complete.
    \item $K^{alg}$ is maximally complete.\qedhere
\end{enumerate}
\end{fact}

It is a classical fact that all extensions of a maximally complete valued field are $vs$-defectless. The following is an analogue in the separable category.

\begin{proposition}\label{prop:SepMaxima->separated}
Assume that $K$ is a separably maximally complete valued field and $L/K$ is a separable extension of valued fields. Then $L/K$ is $vs$-defectless.
\end{proposition}

\begin{proof}
The proof of the main theorem in \cite{delon-separe} may be adapted to our situation. Consider the following properties, where $n\in\mathbb{N}$:
\begin{enumerate}
\item[(a$_n$)] Every $n$-dimensional $K$-subvector space of $L$ has a $K$-valuation basis.
\item [(b$_n$)] If $V$ is an $n$-dimensional $K$-subvector space of $L$ and $c\in L$, then $\{v(c-w) \mid w\in V\}$ admits a maximum.
\end{enumerate}

We will prove both properties by induction on $n$.

First, (a$_1$) holds for every extension by the definition of valuation basis.

We claim that (b$_1$) holds whenever the extension $L/K$ does not contain any proper immediate subextension $L'/K$. Indeed, if $V=c'\cdot K$ for $c'\in L\setminus\{0\}$ and $c\in L\setminus V$, we have
\[
\max\{v(c-w) \mid w\in V\}
=\max \{v(c-dc') \mid d\in K\}
=v(c') + \max \{v(c/c'-d) \mid d\in K\}.
\]
If $c\in V$, equivalently $c/c'\in K$, the result is clear. Else, the proper extension $L'=K(c/c')$ of $K$ is not immediate by assumption, so $\max \{v(c/c'-d) \mid d\in K\}$ exists, which proves the claim. 

Note that if $L/K$ is separable and $K$ is separably maximally complete, then $L/K$ satisfies the assumption of the claim.

Second, we claim that for every extension $L/K$ and every $n$, (a$_n$) together with (b$_n$) imply (a$_{n+1}$). Indeed, if $V$ is of dimension $n+1$, let $V'$ be an $n$-dimensional subspace. By (a$_n$), there is a $K$-valuation basis $(w_1,\cdots,w_n)$ of $V'$. Complete this basis to a $K$-basis of $V$ via $u$. 

By (b$_n$), there is $w'\in V'$ such that 
\[
v(u-w')=\max \{v(u-w) \mid w\in V'\}.
\]
Setting $w_{n+1}:=u-w'$, it is clear that $(w_1,\cdots,w_{n+1})$ is a $K$-valuation basis of $V$.

Finally, assume that $L/K$ is separable and $K$ is separably maximally complete. We claim that under these assumptions (a$_{n+1}$) together with (b$_{n}$) imply (b$_{n+1}$). Indeed, assume for contradiction that $V$ is an ($n+1$)-dimensional $K$-subvector space of $L$ and $c\in L\setminus V$ such that $\{v(c-w) \mid w\in V\}$ admits no maximum. 

Let $(u_1,...,u_{n+1})$ be a $K$-valuation basis of $V$, which exists by (a$_{n+1}$). 
Let $(L',K',v')$ be a $|K|^+$-saturated elementary extension of $(L,K,v)$ (working in the language of pairs of valued fields).

By elementarity, $(u_1,\cdots,u_{n+1})$ is a $K'$-valuation basis of 
\[
V':=\text{span}_{K'}(u_1,\cdots,u_{n+1})\subseteq L',
\]
which is also of dimension $n+1$, as a $K'$-vector space.

By $|K|^+$-saturation of $K'$ and  our assumption that  $\{v(c-w) \mid w\in V\}$ admits no maximum,
there are $r_1,...,r_{n+1}\in K'$ such that 
\[
v'(c-\sum_{1\leq i\leq n+1}r_iu_i) > v(c-w) 
\tag{$*$}
\] 
for any $w\in V$. 

Note that, $K'/K$ is separable, as it is regular, being an elementary extension of fields. Thus $K(r_i)/K$ is not a proper immediate extension. It follows that 
$\{v'(r_i-d) \mid d\in K\}$ admits a maximum, say $v'(r_i-d_i)$ for $d_i\in K$.

Hence $v(c- \sum_{1\leq i\leq n+1} d_iu_i)$ is maximal in $\{v(c-w) \mid w\in V\}$ by the non-Archimedean triangle inequality, ($*$), and the $K'$-valuation independence of $(u_1,\cdots,u_{n+1})$. Indeed, let $w=\sum_{1\leq i\leq n+1}e_iu_i\in V$. Then we compute
\begin{align*}
v(c-w)&=v'\left(\sum_{1\leq i\leq n+1}r_iu_i-w\right)
= v'\left(\sum_{1\leq i\leq n+1}(r_i-e_i)u_i\right)
=\min_{i}\{v'(r_i-e_i)v(u_i)\}\\
&\leq 
\min_{i}\{v'(r_i-d_i)v(u_i)\}
= v'\left(\sum_{1\leq i\leq n+1}(r_i-d_i)u_i\right)\\
&=v'\left(\sum_{1\leq i\leq n+1}r_iu_i-\sum_{1\leq i\leq n+1}d_iu_i\right)
=v\left(c-\sum_{1\leq i\leq n+1}d_iu_i\right).
\end{align*}

This maximality contradicts the assumption.
\end{proof}

The following is a valuative analogue of linear disjointness. 

\begin{definition}[Valuative disjointness]\label{def:valuative-disjoint}
Let $K\subseteq K', L$ be valued subfields of a common valued field. Suppose the extension $L/K$ is $vs$-defectless. We say \emph{$L$ is valuatively disjoint from $K'$ over $K$}, denoted by $L\ind_K^{\vd} K'$, if every $K$-valuation basis of a finite dimensional $K$-vector subspace in $L$ is also $K'$-valuation independent. 
\end{definition}



\begin{proposition}[Algebraic domination]\label{prop:our12.11} Let $\cU\models \SCVF_{p,e}$ be a sufficiently saturated and strongly homogeneous model, where we work in the language $\cLgk^{\lambda}$. Let $K\subseteq F,L\subseteq \cU$ be small $\cLgk^{\lambda}$-substructures. Suppose that 
\begin{itemize}
    \item $K$ is $\VF$-generated and $\VF(K)$, $\VF(F)$, $\VF(L)$ are fields, 
    \item $\VF(L)/\VF(K)$ is $vs$-defectless, 
    \item $\VF(L)\ind^p_{\VF(K)}\VF(F)$ (inside $\VF(\cU)$),
    \item $v(\VF(L))\cap v(\VF(F))=v(\VF(K))$ and 
    \item $\res(\OO(L))\ind_{\res(\OO(K))}^{\ld} \res(\OO(F))$.
\end{itemize}
Let $\sigma$ and $\tau$ be two automorphims of $\cU$ over $K\VG(L)\RF(L)$ with $\sigma(L)=L'$ and $\tau(F)=F'$. Then there is an automorphism $\rho$ of $\cU$ mapping $N=\langle L\cup F\rangle$ to $N'=\langle L'\cup F'\rangle$ and such that $\rho_{|L}=\sigma_{|L}$ and $\rho_{|F}=\tau_{|F}$. Furthermore, $\mathrm{Frac}(\VF(N))=\VF(L)\VF(F)$, $\VG(N)= \VG(L)+\VG(F)$, $\RF(N)$ is contained in $\RF(F)\RF(L)$, and we have $\VF(L)\ind_{\VF(K)}^{\vd} \VF(F)$.   
\end{proposition}

\begin{proof}
This is proved as the corresponding result \cite[Proposition~6.3.1]{CuHiYe26} for $\ACVF$, using quantifier elimination in $\SCVF_{p,e}$ (Proposition~\ref{prop:QE-SCVF}) and noting that as $\VF(\cU)$ is a separable extension of $\VF(K)$, $\VF(L)$ and $\VF(F)$, the assumption $\VF(L)\ind^p_{\VF(K)}\VF(F)$ implies that $\VF(\cU)/\VF(L)\VF(F)$ is separable as well, thus $(\mathrm{Frac}(\VF(N)),\VG(N),\RF(N))$ is an $\cLgk^{\lambda}$-substructure of $\cU$.
\end{proof}

As in the case of $\ACVF$ (see \cite[Proposition~6.3.3]{CuHiYe26}, and \cite{CD}, for types from the $\VF$-sort alone), we may infer an algebraic chacterization of definable types as well as a domination statement from this.

\begin{proposition}\label{lem:domination} Let $\cU\models \SCVF_{p,e}$. Let $K\subseteq L\subseteq \cU$ be $\cLgk^{\lambda}$-substructures with $K\preccurlyeq \cU$ and $\VF(L)$ a field. Then 
\begin{enumerate}[label=(\arabic*)]
    \item $\tp(L/K)$ is definable if and only if $\VF(L)/\VF(K)$ is $vs$-defectless and $\tp(\VG(L)/\VG(K))$ is definable;
    \item if $tp(L/K)$ is definable and $K\subseteq K'\subseteq \cU$, with $K'$ an $\cLgk^{\lambda}$-substructure of $\cU$ such that $\VF(K')$ is a field, then $\tp(L/K)\mid K'$ is implied by
    
\begin{equation}\label{eq:domination}\tag{$\ast$}
\tp(L/K)\cup\tp(\RF(L)/\RF(K))\mid \RF(K')\cup\tp(\VG(L)/\VG(K))\mid \VG(K') \cup\{\VF(L)\ind^p_{\VF(K)}\VF(K')\}. 
\end{equation}
Moreover, if $\widetilde{L}\models \tp(L/K)\mid K'$, then $(\VF(\widetilde{L})\VF(K'),\VG(\widetilde{L})+\VG(K'),\RF(\widetilde{L})\RF(K')$ is an $\cLgk^{\lambda}$-substructure of $\cU$.
\end{enumerate}
\end{proposition}

\begin{remark}
By pure stable embeddedness of the residue field and the value group (Proposition~\ref{prop:QE-SCVF}), it amounts to the same to consider the types $\tp(\RF(L)/\RF(K))$ and $\tp(\VG(L)/\VG(K))$ in the languages of rings and ordered groups, respectively, or in the full ambient valued field language. As all types in $\ACF_p$ are definable, there is no condition in the statement about the definability of $\tp(\RF(L)/\RF(K))$.
\end{remark}

\begin{proof}
Let us first prove (1). If $\tp(L/K)$ is definable, so is $\tp(\VG(L)/\VG(K))$ (in the language of ordered groups), by pure stable embeddedness of the value group. 

\begin{claim}\label{claim:sep-max-elem}
$K$ admits an immediate elementary extension $K_1$ which is separably maximally complete. 
\end{claim}

To prove the claim, let $\widetilde{K}$ be a maximal immediate extension of $K$. Then $\widetilde{K}$ is maximally complete, and a model of $\ACVF$. If $\mathcal{B}$ is a transcendence basis of $\widetilde{K}$ over $K$ and $K_0:=K(b^{p^{-n}}:b\in\mathcal{B},n\geq0)$, it is straight forward to check that $K_1:=K_0^{sep}$ is an elementary extension of $K$. As $K_1^{alg}=\widetilde{K}$, the valued field $K_1$ is separably maximally complete by Fact~\ref{fact:char-sep-max}. This proves the claim.

Clearly, $\tp(L/K)\mid K_1\vdash \VF(L)\ind^{ld,p}_{\VF(K)}\VF(K_1)$, and so the compositum $\VF(L)\VF(K_1)$ is a separable etxension of $\VF(K_1)$. Now let $V$ be a finite-dimensional $\VF(K)$-subvector space of $\VF(L)$. Consider the base change $V_1=\VF(K_1)V$. Then $\dim_{\VF(K_1)} V_1=\dim_{\VF(K)} V$, and by Proposition~\ref{prop:SepMaxima->separated}, $V_1$ admits a $\VF(K_1)$-valuation basis. As $\tp(L/K)$ is definable, it follows that $V$ admits a $\VF(K)$-valuation basis, too. This shows that $L/K$ is $vs$-defectless.

The converse direction in (1) follows as in the proof of \cite[Theorem 5.9]{cubi-ye}, where the assumption that $L$ is a model is not used. Note that in the stable reduct $\SCF_{p,e}$ (which admits quantifier elimination in $\cLr^\lambda$ by Fact~\ref{F:QE-SCF}), $p_0:=\tp_{\cLr^\lambda}(\VF(L)/\VF(K))$ is definable, and if $\VF(L)\models p_0\mid\VF(K')$ for $K\elres K'\elres\cU$, then $\VF(L)\ind^{ld,p}_{\VF(K)}\VF(K')$, so $\VF(L)\VF(K')$ is an $\cLr^\lambda$- structure. 

Part (2) follows as in the proof of \cite[Proposition~6.3.3]{CuHiYe26}, using Proposition~\ref{prop:our12.11}. The moreover part follows from the end of the previous paragraph.
\end{proof}

We now state a variant we will need later.

\begin{corollary}
    \label{Cor:our12.11-types} Let $\cU\models \SCVF_{p,e}$, and et $K\subseteq F,L\subseteq \cU$ be $\cLgk^{\lambda}$-substructures. Suppose that 
\begin{itemize}
    \item $K$ is $\VF$-generated and $\VF(K)$, $\VF(F)$, $\VF(L)$ are fields, 
    \item $\VF(L)/\VF(K)$ is $vs$-defectless,   
      \item $\VF(L)\ind^p_{\VF(K)}\VF(F)$ (inside $\VF(\cU)$),
    \item $v(\VF(L))\cap v(\VF(F))=v(\VF(K))$ and 
    \item $\res(\OO(L))\ind_{\res(\OO(K))}^{\ld} \res(\OO(F))$.
\end{itemize}

Then $\tp(L/K)\cup\tp(\RF(L)/\RF(F))\cup\tp(\VG(L)/\VG(F))\cup\{\VF(L)\ind^p_{\VF(K)}\VF(F)\}\vdash\tp(L/F)$.
\end{corollary}

\begin{proof}
    This is a consequence of Proposition~\ref{prop:our12.11}.
\end{proof}

\section{Separable vs-defectless pairs of separably closed valued fields}\label{S:sep-vs-defectless}
In this section we study separable $vs$-defectless pairs of separably closed valued fields.

\begin{definition}\label{def:pairs-SCVF}
\begin{itemize}
    \item We set $\cL:=\cLgk^{\lambda}$, and $\cL_P:=\cL\cup\{P\}$ denotes the corresponding language of pairs. (Formally, we add a new unary predicate to each sort of $\cLgk^{\lambda}$, but by abuse of notation we will denote this new predicate by $P$ in each sort.)
    \item Fix a prime $p$, a pair $E=(e,e_P)$ of elements from $\mathbb{N}\cup\{\infty\}$, and let $T_{\RF}$ and $T_{\VG}$ be completions of the theory of stably embedded pairs of $\ACF_p$ and $\DOAG$, respectively. Let $\cK_{\RF}$ (resp. $\cK_{\VG}$) be the natural class such that $T_{\RF}$ (resp. $T_{\VG}$) is the theory of $\cK_{\RF}$-beautiful pairs (resp. $\cK_{\VG}$-beautiful pairs). We assume that at least one of $\cK_{\RF}$ and $\cK_{\VG}$ is not the trivial class.
\item  We let $\SCVF_{p,E}(T_{\RF},T_{\VG})$ be the $\cL_P$-theory of pairs $(K,P(K))$ of models of $\SCVF_p$ satisfying the following:
\begin{itemize}
        \item $P(K)\models\SCVF_{p,e_P}$, $K\models\SCVF_p$ and $e(\VF(K)/\VF(P(K))=e$,
        \item $\VF(K)/\VF(P(K))$ is separable,
        \item $\VF(K)/\VF(P(K))$ is $vs$-defectless, and
        \item $(\VG(K),\VG(P(K)))\models T_{\VG}$ and $(\RF(K),\RF(P(K)))\models T_{\RF}$.
    \end{itemize}
\item We define $\cK_{p,E}(\cK_{\VG},\cK_{\RF})$ to be the class of $\cL_P$-structures $(B,P(B))$ satisfying the following:
      \begin{itemize}
      \item $B$ is an $\cLgk^{\lambda}$-substructure of a model of $\SCVF_p$,
        \item $P(B)\models\SCVF_{p,e_P}$ and $e(\mathrm{Frac}(\VF(B))/\VF(P(B))\leq e$,
        \item $\mathrm{Frac}(\VF(B))/\VF(P(B))$ is separable,
        \item $\mathrm{Frac}(\VF(B))/\VF(P(B))$ is $vs$-defectless, and
        \item $(\VG(B),\VG(P(B)))\in \cK_{\VG}$ and $(\RF(B),\RF(P(B)))\in \cK_{\RF}$.
    \end{itemize}
    \item Let $(B,P(B)),(C,P(C))\in \cK_{p,E}(\cK_{\VG},\cK_{\RF})$ such that $(B,P(B))\subseteq(C,P(C))$. We say that $(B,P(B))$ is a \emph{$\BP$-substructure} of $(C,P(C))$, and we write $(B,P(B))\subseteq_{\BP}(C,P(C))$, if both $(\VG(B),\VG(P(B)))\subseteq_{\BP}(\VG(C),\VG(P(C)))$ and $(\RF(B),\RF(P(B)))\subseteq_{\BP}(\RF(C),\RF(P(C)))$. 
    
    A \emph{$\BP$-embedding} is an isomorphism onto a $\BP$-substructure.
\end{itemize}
\end{definition}

\begin{lemma}\label{lem:consistent} For any $E=(e_P,e)$, $T_{\RF}$ and $T_{\VG}$ as in Definition~\ref{def:pairs-SCVF}, the theory $\SCVF_{p,E}(T_{\RF},T_{\VG})$ is consistent.
\end{lemma}

\begin{proof}
    We will construct a model of $\SCVF_{p,E}(T_{\RF},T_{\VG})$. Let $(k,P(k))\models T_{\RF}$ and $(\Gamma,P(\Gamma))\models T_{\VG}$ be $\omega$-saturated models. We consider the Hahn series fields (endowed with the $t$-adic valuation) 
    $$P_0:=P(k)((t^{P(\Gamma)}))\subseteq K_0:=k((t^{\Gamma})).$$ Then $P_0\elres K_0\models \ACVF$, and both $\trdeg(P_0)$ and $\trdeg(K_0/P_0)$ are infinite. Let $\mathcal{B}_0$ be a transcendence basis of $P_0$ and $\mathcal{C}_0$ a transcendence basis of $K_0/P_0$. Choose subfamilies $\mathcal{B}_1\subseteq\mathcal{B}_0$ and $\mathcal{C}_1\subseteq\mathcal{C}_0$ such that $|\mathcal{B}_1|=e_P$ if $e_P$ is finite, else $|\mathcal{B}_1|=\aleph_0$, and similarly $|\mathcal{C}_1|=e$ if $e$ is finite, else $|\mathcal{C}_1|=\aleph_0$. Set 
    $$P_1:=\mathbb{F}_p(\mathcal{B}_1\cup\{b^{p^{-n}}\mid b\in\mathcal{B}_0\setminus\mathcal{B}_1,n\in\mathbb{N}\})^{\sep}$$
    and similarly
    $$K_1:=P_1(\mathcal{C}_1\cup\{b^{p^{-n}}\mid b\in\mathcal{C}_0\setminus\mathcal{C}_1,n\in\mathbb{N}\})^{\sep}.$$
We have the following:
\begin{enumerate}
    \item $P_1^{\alg}=P_0$ and $K_1^{\alg}=K_0$.
        \item $P_1\models \SCVF_{p,e_P}$, $K_1\models \SCVF_p$ and $e(K_1/P_1)=e$.
    \item $K_1/P_1$ is separable.
    \item $\RF(K_1)=k$, $\RF(P_1)=P(k)$, $\VG(K_1)=\Gamma$ and $\VG(P_1)=P(\Gamma)$.
    \item $P_1$ is separably maximally complete.
    \item $K_1/P_1$ is $vs$-defectless.
\end{enumerate}
Indeed, (1)-(3) follow easily from the construction. (4) is a consequence of density of $K_1$ in $K_0$ and of $P_1$ in $P_0$, which hold by (1) in combination with Fact~\ref{F:dense}. (5) is by Fact~\ref{fact:char-sep-max} together with the first part of (1). Finally, (6) is a consequence of (5) together with Proposition~\ref{prop:SepMaxima->separated}.

Thus, $(K_1,P_1)\models   \SCVF_{p,E}(T_{\RF},T_{\VG})$. 
\end{proof}

\begin{remark}\label{rem:plus}
If $(K,P(K))\models   \SCVF_{p,E}(T_{\RF},T_{\VG})$ for $E=(e_P,e)$, then $K\models\SCVF_{p,e_P+e}$.\footnote{Here, by convention $\infty+n:=n+\infty:=\infty+ \infty:=\infty$ for all $n\in\mathbb{N}$.}  \qedhere
\end{remark}

\begin{lemma}[Main Embedding Lemma]\label{lem:Embedding}
Assume that $(K,P(K))$ is an $\aleph_1$-saturated model of $\SCVF_{p,E}(T_{\RF},T_{\VG})$. Let $(B,P(B))\subseteq_{\BP} (C,P(C))$ with $(B,P(B)),(C,P(C))\in \cK_{p,E}(\cK_{\VG},\cK_{\RF})$ and $C$ countable. Then every $\BP$-embedding $f:(B,P(B))\hookrightarrow (K,P(K))$ extends to a $\BP$-embedding $g:(C,P(C))\hookrightarrow (K,P(K))$.
\end{lemma}


\begin{proof}[Proof of Lemma~\ref{lem:Embedding}] 
We may argue exactly as in the proof of \cite[Theorem~7.1.5]{CuHiYe26}, so we will be brief.

\emph{Step 0.} Without loss of generality, we may suppose $(B,P(B))\subseteq_{\BP} (K,P(K))$ and that the valued field sorts of $B$ and $C$ are fields. We may also assume $P(B)=P(C)$, as by the saturation assumption the elementary $\cL$-embedding $f\upharpoonright_{P(B)}:P(B)\hookrightarrow P(K)$ extends to an elementary embedding $\tilde{f}_P:P(C)\hookrightarrow P(K)$, and by Proposition~\ref{lem:domination}(2) (working in the theory $\SCVF_{p,e_P+e}$, see Remark~\ref{rem:plus}), for $(B,P(B))\subseteq (\langle B \cup P(C)\rangle, P(C))\subseteq (C,P(C)$ one has $(B,P(B))\subseteq_{\BP} (\langle B \cup P(C)\rangle, P(C))\subseteq_{\BP} (C,P(C))$, and $\qftp_{\cL_P}((\langle B \cup P(C)\rangle, P(C)))$ is completely determined by $\qftp_{\cL_P}(B,P(B))$, $\tp_{\cL}(P(C)$ and the fact that $(B,P(B)\subseteq_{\BP}(C,P(C))$, yielding that $f\cup \tilde{f}_P$ uniquely extends to a $\BP$-embedding $\tilde{f}:(\langle B \cup P(C)\rangle, P(C))\hookrightarrow (K,P(K))$. 

\medskip

\emph{Step 1.} Since $(C,P(C))\in \cK_{p,E}(\cK_{\VG},\cK_{\RF})$, we have that $(\RF(C), \RF(P(C))\in \cK_{\RF}$ and similarly $(\VG(C),\VG(P(C))\in \cK_{\VG}$. By beauty transfer of $\cK_{\RF}$ and beauty transfer of $\cK_{\VG}$, $(\RF(K),\RF(P(K)))$ is $\aleph_1$-$\cK_{\RF}$-beautiful and $(\VG(K),\VG(P(K))$ is $\aleph_1$-$\cK_{\VG}$-beautiful. In particular, there is an $(\cL_\Ring)_{\BP}$-embedding $g_{\RF}\colon (\RF(C), \RF(P(C))\to (\RF(K),\RF(P(K))$ over $(\RF(B),\RF(P(B))$, and similarly, there is an $(\cL_\og)_{\BP}$-embedding  $g_{\VG}\colon (\VG(C),\VG(P(C)))\to (\VG(K),\VG(P(K)))$ over $(\VG(B),\VG(P(B)))$. 

As $g_{\RF}$ and $g_{\VG}$ are $\BP$-embeddings, $g_{\RF}(\RF(C))\ind_{\RF(P(B))}^{\ld} \RF(P(K))$ and $g_{\VG}(\VG(C))\cap \VG(P(K))=\VG(P(B))$ (see \cite[Fact~6.2.1 and Fact~5.0.1]{CuHiYe26}).

\medskip

\emph{Step 2.} As $e(\VF(C)/\VF(P(C))\leq e$ by assumption, the saturation of $(K,P(K))$ yields that $e(\VF(K)/\VF(B)\VF(P(K))\geq e(\VF(C)/\VF(B))$. The proof of quantifier elimination for $\SCVF_{p,e_P+e}$ in $\cL$ (proof of Proposition~\ref{prop:QE-SCVF}), combined with $\aleph_1$-saturation of the pair $(K,P(K))$ implies that there is an $\cL$-embedding $g\colon (C,P(C))\to (K,P(K))$ over $B$ inducing the maps $g_{\RF}$ and $g_{\VG}$ from Step 1 such that $g(\VF(C))\ind^{p}_{\VF(B)}\VF(P(K))$, thus also $g(\VF(C))\ind^{p}_{\VF(P(B))}\VF(P(K))$. (Note that being $p$-independent over $\VF(B)\VF(P(K))$ is expressible by a partial $\cL_P$-type over $B$.)

Moreover, we have that
\begin{align*}
& g(C)\models \tp(C/P(B)) \\ 
& g(\RF(C))\models  \tp(\RF(C)/\RF(P(B)))\mid\RF(P(K)) \\ 
& g(\VG(C))\models  \tp(\VG(C)/\VG(P(B)))\mid\VG(P(K)).  
\end{align*}
Part (2) of Proposition~\ref{lem:domination} yields that $g$ induces a $\BP$-embedding. (It is clear that $g$ respects the predicate $P$ and is thus an $\cL_P$-embedding.) \end{proof}

\begin{theorem}\label{thm:SCVF-Pair-Complete}
For any $E=(e_P,e)$, $T_{\RF}$ and $T_{\VG}$ as in Definition~\ref{def:pairs-SCVF} we have the following.

\begin{enumerate}
    \item The theory $\SCVF_{p,E}(T_{\RF},T_{\VG})$ is complete.
    \item Any $\BP$-isomorphism between $\BP$-substructures of models of $\SCVF_{p,E}(T_{\RF},T_{\VG})$ is elementary.
    \item In any model of $\SCVF_{p,E}(T_{\RF},T_{\VG})$, the predicate $P$ is stably embedded with induced structure a pure model of $\SCVF_{p,e_P}$.
\end{enumerate}
\end{theorem}

\begin{proof}
    By Lemma~\ref{lem:consistent} the theory $\SCVF_{p,E}(T_{\RF},T_{\VG})$ is consistent. It is easy to see that any $\BP$-substructure of a model of $\SCVF_{p,E}(T_{\RF},T_{\VG})$ is a filtered union of countable $\BP$-substructures. Part (2) thus follows from from Lemma~\ref{lem:Embedding}. Given $M_0\models\SCVF_{p,e_P}$, the trivial pair $(M_0,M_0)$ (with $P(M_0)=M_0)$ may be $\BP$-embedded into any $|M_0|$-saturated model of $\SCVF_{p,E}(T_{\RF},T_{\VG})$. This yields (1). The straight forward proof of (3) is left to the reader.
\end{proof}
    



\section{Beautiful pairs of separably closed valued fields}\label{S:SCVF-BP}
In this section, we specialize the results from the previous section and thereby infer results on beautiful pairs of separably closed valued fields.

In what follows, we keep the conventions and the notations from Definition~\ref{def:pairs-SCVF}. By $\cK^{\DE}_{\VG}$ we denote the natural class of all definable types in $\DOAG$, by $\cK^{\bdd}_{\VG}$ that of all bounded definable types, and by $\cK^{\triv}_{\VG}$ the trivial natural class in $\DOAG$. The classes $\cK^{\DE}_{\RF}$ and $\cK^{\triv}_{\RF}$ in $\ACF_p$ are similarly defined.

\begin{lemma}
Let $e_P\in\mathbb{N}\cup\{\infty\}$ if $e=\infty$, and $e_P=0$ if $e<\infty$. Assume that $\cK_{\RF}$ and $\cK_{\VG}$ correspond to completions of the theory of stably embedded pairs of $\ACF_p$ and $\DOAG$, respectively, where not both are the trivial class. Then the following holds.

\begin{enumerate}
    \item $\cK_{p,(e,e_P)}(\cK_{\VG},\cK_{\RF})$ is a natural class of definable types in the theory $\SCVF_{p,e}$. 
\item If $e<\infty$, in $\SCVF_{p,e}$ the natural classes $\cK_{p,(e,0)}(\cK^{\DE}_{\VG},\cK^{\DE}_{\RF})$, $\cK_{p,(e,0)}(\cK^{\bdd}_{\VG},\cK^{\DE}_{\RF})$ and $\cK_{p,(e,0)}(\cK^{\triv}_{\VG},\cK^{\DE}_{\RF})$ correspond to the classes of all definable types, of all bounded definable types (i.e., those orthogonal to $p_{+\infty}$ in $\VG$) and of all stably dominated types. 

Moreover, the latter class is equal to the class of all definable types orthogonal to $\VG$ and to the class of all generically stable types.
\item If $e=\infty$, in $\SCVF_{p,e}$ the natural classes $\cK_{p,(e,e_P)}(\cK^{\DE}_{\VG},\cK^{\DE}_{\RF})$, $\cK_{p,(e,e_P)}(\cK^{\bdd}_{\VG},\cK^{\DE}_{\RF})$ and $\cK_{p,(e,e_P)}(\cK^{\triv}_{\VG},\cK^{\DE}_{\RF})$ correspond to the classes of all definable types, of all bounded definable types and of all definable types orthogonal to $\VG$, in each case with the condition that the $\VF$-extension given by a realization of the type is of imperfection degree at most $e_P$. 

Moreover, the class of all definable types orthogonal to $\VG$ is equal to the class of all generically stable types.
\end{enumerate}
\end{lemma}

\begin{proof}
    By Proposition~\ref{lem:domination}, the types in the classes in question are all definable. Note that if $e<\infty$ and $P(K)\elres K\models\SCVF_{p,e}$, then $e(\VF(K)/\VF(P(K)))=0$, which imposes $e_P=0$ in this case. 

    Part (1) is proved as in the corresponding cases in $\ACVF$ (see \cite[Lemma~7.1.3]{CuHiYe26}, noting that for closure under base extension we may argue as in \cite[Remark~7.1.2]{CuHiYe26}, using in addition the moreover part in Proposition~\ref{lem:domination}(2). 

    Part (2) and (3) are straight forward, except the moreover parts. Note that in any theory stably dominated types are generically stable, and generically stable types are orthogonal to any interpretable stably embedded set $D$ such that $\Th(D)$ eliminates imaginaries and is equipped with a definable total order. The latter holds for $\VG$ in $\SCVF$, as it is stably embedded and a pure model of $\DOAG$.

    If $e<\infty$, $K\elres\cU$ and if $K\subseteq L,K'\subseteq\cU$ are $\cLgk^{\lambda}$-substructures such that $\VF(L)$ and $\VF(K')$ are fields, then $\VF(L)\ind^p_{\VF(K)}\VF(L)$ is automatically satisfied, as $e(\VF(L)/\VF(K))=0$. Thus if $\tp(L/K)$ is definable and orthogonal to $\VG$, it is stably dominated by Proposition~\ref{lem:domination}(2), as $\RF$ is stably embedded with induced structure a pure model of $\ACF_p$. This finishes the proof of the moreover part in (2).

    If $e=\infty$ and $p(x):=\tp(L/K)$ is definable and orthogonal to $\VG$, then Proposition~\ref{lem:domination}(2) implies that $p(x)\otimes p(y)=p(y)\otimes p(x)$, as both independence in $\RF$ and $\ind^p$ are symmetric. This shows that $p$ is generically stable, as $\SCVF_{p,\infty}$ is NIP (see \cite{simon}).
\end{proof}

We conjecture that the stably dominated types in $\SCVF_{p,\infty}$ are precisely the types in the class $\cK_{p,(e,0)}(\cK^{\triv}_{\VG},\cK^{\DE}_{\RF})$. This is easily seen to be equivalent to Conjecture~\ref{Conj:stdom} below.  Contrarily to the case of finite imperfection degree $e$, where the imaginaries in $\SCVF_{p,e}$ are classified by the geometric sorts (see \cite{HKR2016}), no classification of imaginaries in $\SCVF_{p,\infty}$ is known, and so an approach by inspection in the (geometric) sorts is not possible. 

\begin{conjecture}\label{Conj:stdom}
   Let $K\models\SCVF_{p,\infty}$, and let $q(x)$ be the definable type of an element $x\in\OO(\cU)$  such that $\res(x)$ is generic over $\RF(K)$ and $x$ is $p$-independent in $\VF(\cU)$ over $\VF(K)$. Then $q$ is orthogonal to $\VF$ (equivalently, generically stable). We conjecture that $q$ is not stably dominated. 
\end{conjecture}

\begin{remark}
 Conjecture~\ref{Conj:stdom} implies that $\SCVF_{p,\infty}$ is not metastable over $\VG$. Observe that by \cite[Corollary~6.11]{HKR2016}  $\SCVF_{p,e}$ is metastable over $\VG$ for any finite degree of imperfection $e$.   
\end{remark}

\begin{theorem}\label{thm:beauty-SCVF}
Let $e_P\in\mathbb{N}\cup\{\infty\}$ if $e=\infty$, and $e_P=0$ if $e<\infty$. Assume that $\cK_{\RF}$ and $\cK_{\VG}$ correspond to completions $T_{\RF}$ and $T_{\VG}$ of the theory of stably embedded pairs of $\ACF_p$ and $\DOAG$, respectively, where not both are the trivial class. 

Set $T:=\SCVF_{p,e}$ and $\cK:=\cK_{p,(e,e_P)}(\cK_{\VG},\cK_{\RF})$. Then $\cK$ has AP, EP and beauty transfer, and $T_{\BP}(\cK)=\SCVF_{p,(e,e_P)}(T_{\RF},T_{\VG})$. Moreover, $\SCVF_{p,e}$ has UDDT($\cF_{\cK}$).
\end{theorem}

\begin{proof}
The proof is completely analogous to the proof of the corresponding result for ACVF (\cite[Theorem~7.1.5]{CuHiYe26}). 

Existence of $\cK$-beautiful pairs, beauty transfer for $\cK$ and $T_{\BP}(\cK)=\SCVF_{p,(e,e_P)}(T_{\RF},T_{\VG})$ are direct consequences of Lemma~\ref{lem:Embedding}. AP for $\cK$ follows from the existence of $\cK$-beautiful pairs, and the fact that $\cK$-beautiful pairs are elementary pairs implies EP. Finally, UDDT($\cF_{\cK}$) is by Fact~\ref{fact:main-beauty}(1).
\end{proof}

Combining Proposition~\ref{lem:domination} with Theorem~\ref{thm:beauty-SCVF}, we obtain the following analog of the corresponding result in $\ACVF$ (see Example~\ref{Ex:ACVF}).

\begin{corollary}\label{cor:completions-SE}
Let $(\SCVF_{p,e})_{\SE}$ be the $\cL_P$-theory of stably embedded elementary pairs of models of $\SCVF_{p,e}$. 
\begin{enumerate}
    \item If $e<\infty$, $(\SCVF_{p,e})_{\SE}$ has 8 completions, namely the theory of the trivial pair and the theories $\SCVF_{p,(e,0)}(T_{\RF},T_{\VG})$, for $T_{\RF}$ and $T_{\VG}$ not both trivial. Each completion corresponds to the theory of $\cK$-beautiful pairs for some natural class $\cK$ having AP, EP and beauty transfer.
    \item The completions of $(\SCVF_{p,e})_{\SE}$ are given by the theory of the trivial pair and by the theories $\SCVF_{p,(\infty,e_P)}(T_{\RF},T_{\VG})$, for $e_P\in\mathbb{N}\cup\{\infty\}$, and $T_{\RF}$ and $T_{\VG}$ not both trivial. Each completion corresponds to the theory of $\cK$-beautiful pairs for some natural class $\cK$ having AP, EP and beauty transfer.\qedhere
\end{enumerate}
\end{corollary}

\begin{corollary}\label{Cor:strictpro}
    Let $\cK=\cK_{p,(e,e_P)}(\cK_{\VG},\cK_{\RF})$ be a natural class of definable types in the theory $\SCVF_{p,e}$, and let $X$ be an $A$-definable set, where  $A\subseteq\cU\models\SCVF_{p,e}$. Then the space $\cF_X$ carries naturally the structure of a strict $A$-pro-definable set. In particular, the space $S^{\DE}_{X}$ of all definable types, the space $\widetilde{X}$ of all bounded definable types and the space $\widehat{X}$ of all generically stable definable types are strict pro-definable.\qedhere
\end{corollary}

\begin{remark}
    In finite degree of imperfection, strict pro-definability of $\widehat{X}$ has been established in \cite{HKR2016} with completely different methods. All the other results on strict pro-definability are new. Corollary~\ref{cor:completions-SE} together with Corollary~\ref{Cor:strictpro} give a complete answer to \cite[Question~9.3.6]{CuHiYe26}, in arbitrary degree of imperfection.
\end{remark}

\begin{remark}
    As in the case of ACVF (see \cite[Section~7.3]{CuHiYe26}), one may infer a result on $\cK$-beautiful pairs for classes $\cK=\cK(\cK_{\VG},\cK^{\DE}_{\RF})$ considered in the language of the geometric sorts $\cL^{\mathcal{G}}$. Indeed, if $\cK$ is the class of $\sepa$-pairs in $\cL^{\mathcal{G}}_P$ induced by $\cK_{\VG}$ and $\cK^{\DE}_{\RF}$, then $\cK$-beautiful pairs exist and are elementary pairs, axiomatized by $\SCVF_{p,(e,0)}(T_{\VG},(\ACF_p)_{\BP})$ for $e<\infty$ and by $\SCVF_{p,(\infty,\infty)}(T_{\VG},(\ACF_p)_{\BP})$, respectively. As a consequence, for any definable set $X$ in the geometric sorts, 
    $\cF_X$ is strict pro-definable, where $\cF=\cF_{\cK}$.
\end{remark}

\begin{proof}
    The proof is by resolutions, as in the case of ACVF (see \cite[Theorem~7.3.1]{CuHiYe26}). If $K\elres\cU$ is countable and $\cU$ is $\aleph_1$-saturated, by density of $\VF(\cU)^{p^{\infty}}$ in $\VF(\cU)$ (see Fact~\ref{F:dense}) and saturation there is $c\in \OO\cap\cU^{p^{\infty}}$ such that $\res(c)\not\in \RF(K)$. Note that $\tp(c/K)$ is definable, and actually stably dominated. 
    
    If in addition $K$ is $\omega$-saturated, $\VF(K)^{p^{\infty}}$ is dense in $\VF(K)$, and so $S_n(\VF(K))=S_n(\VF(K)^{p^{\infty}}$ 
     by \cite[Lemma~2.1]{HKR2016}, i.e., every $\OO$-lattice in $\VF(K)^n$ admits a basis of vectors with coefficients from $\VF(K)^{p^{\infty}}$. 

One may now show that every $\aleph_1$-saturated model of $\SCVF_{p,(e,0)}(T_{\VG},(\ACF_p)_{\BP})$, considered in the language $\cL^{\mathcal{G}}_P$ is a $\cK$-beautiful pair, by reducing this result to the three-sorted language as in \cite[Theorem~7.3.1]{CuHiYe26}. We leave the details to the reader.  
\end{proof}

\begin{corollary}
    For any $e\in\mathbb{N}\cup\{\infty\}$, the theory $\left(\SCVF_{p,e}\right)_{\SE}$ is NIP.
\end{corollary}

\begin{proof}
A theory is NIP if and only if all its completions are NIP. The result thus follows from Corollary~\ref{cor:completions-SE}, using that $\SCVF_{p,e}$ is NIP and applying \cite[Corollary~2.4.17]{CuHiYe26}.
\end{proof}

\begin{remark}
    More generally, one may show that $\SCVF_{p,E}(T_{\RF},T_{\VG})$ is NIP for any choice of $E=(e_P,e)$, $T_{\RF}$ and $T_{\VG}$.
\end{remark}

\bibliography{biblio}

\end{document}